\documentclass[12pt,preprint]{article}

\usepackage{natbib}

\usepackage{soul,xcolor}
\setstcolor{red}

\usepackage{amssymb}
\usepackage{amsmath}

\usepackage{antpolt}
\usepackage{mathpazo} % Palatino
\usepackage{avant}  

\usepackage{hyperref}

\usepackage{amsmath, amsthm, amsfonts,stmaryrd,amssymb}
\usepackage[top=1in, bottom=1.25in, left=1.25in, right=1.25in]{geometry}
\usepackage{mathrsfs}    
\usepackage{dsfont }                  % Nice \mathscr fonts
\usepackage{mathtools}                   % \mathrlap
\usepackage{booktabs}
\usepackage{tikz}
\usepackage[all]{xy}
\usepackage[utf8]{inputenc}
\usepackage{mathtools}                     % Extendable arrows with text

\usepackage{faktor}
\usepackage{ marvosym }
\usepackage{bbm}                         % For \mathbbm{1}
\usepackage{framed}
\usepackage{graphicx} 
 
\usetikzlibrary{shapes.misc}
\tikzset{cross/.style={cross out, draw=black, minimum size=2*(#1-\pgflinewidth), inner sep=0pt, outer sep=0pt},
cross/.default={1pt}}
\usepackage{xcolor}
\newif\ifcomments
\commentstrue

\definecolor{mycolor}{RGB}{20, 20, 122}

\newcommand{\mres}{\mathbin{\vrule height 1.6ex depth 0pt width
0.13ex\vrule height 0.13ex depth 0pt width 1.3ex}}

\usepackage{theoremref}

\numberwithin{equation}{section}
\theoremstyle{plain}

\newtheorem{remark}{Remark}
\newtheorem{theorem}{Theorem}

\newtheorem{proposition}[theorem]{Proposition}
\def\ep{\varepsilon}

\def\si{\sigma}

\def\X{{\mathcal X}}
\def\Y{{\mathcal Y}}

\def\Sc{{\mathcal S}}

\def\P{{\mathcal P}}

\def\L{\mathcal{L}}
\def\Ec{\mathcal{E}}
\def\Dc{\mathcal{D}}

\def\R{{\mathbb R}}

\def\Cc{{\mathcal C}}
\def\Pc{{\mathcal P}}
\def\Kc{{\mathcal K}}

\def\J{{\mathcal J}}

\def\Id{\operatorname{Id}}

\newcommand{\SI}[2]{ \int   #1  \,  #2  }

\def\u{{\varphi}}
\def\v{{\psi}}
\def\ue{\varphi_\ep}
\def\ve{\psi_\ep}

\def\us{u_\ep}

\def\he{h_\ep}

\def\uek{\varphi_\ep^k}

\def\uekp{\varphi_\ep^{k+1}}
\def\vekp{\psi_\ep^{k+1}}

\def\Ke{K_\ep}

\usepackage{comment}
\usepackage{tikz-cd}

\usepackage{enumerate} 
\newcommand{\pp}[2]{\frac{\partial #1}{\partial #2}}

\DeclareMathOperator{\diff}{d}

\DeclareMathOperator{\OT}{OT}

\newtheorem{assumption}{Assumption}
 
\newif\ifsolns
\solnsfalse

\ifsolns
\newcommand{\soln}[1]{\newline \noindent {\bfseries Solution:} {\itshape #1}}
\else
\newcommand{\soln}[1]{}
\fi

\usepackage{ifthen}

\newboolean{colorText}
\setboolean{colorText}{true} % change to'false' to desactivate the coloration

\newcommand{\changes}[1]{%
\ifthenelse{\boolean{colorText}}{\textcolor{blue}{#1}}{#1}%
}

\usepackage{caption}
\usepackage{subcaption}

\newcommand{\norm}[1]{\lVert #1 \rVert}

\begin{document}

%%%%%%%%

\title{Entropic approximations of the semigeostrophic shallow water
  equations}

\author{
    Jean-David Benamou \\
    \textsc{Inria Paris}, MOKAPLAN Team \\
    48 Rue Barrault\\
    75013 Paris, France \\
    \texttt{jean-david.benamou@inria.fr}
    \and 
    Colin J. Cotter \\
    \textsc{Imperial College London}, Department of Mathematics \\
    South Kensington Campus \\
    London SW7 2AZ, United Kingdom  \\
    \texttt{colin.cotter@imperial.ac.uk}
    \and 
    Jacob J.M. Francis \\
     \textsc{Imperial College London}, Department of Mathematics \\
    South Kensington Campus \\
    London SW7 2AZ, United Kingdom  \\
    \texttt{jacob.francis18@imperial.ac.uk}
    \and 
    Hugo Malamut \\
    \textsc{Inria Paris}, MOKAPLAN Team \\
    48 Rue Barrault\\
    75013 Paris, France
    \texttt{hugo.malamut@inria.fr}
}

\maketitle

%\begin{history}
%\received{(Day Month Year)}
%\revised{(Day Month Year)}
%\accepted{(Day Month Year)}
%\comby{(xxxxxxxxxx)}
%\end{history}

\begin{abstract}
  We develop a discretisation of the semigeostrophic rotating shallow
  water equations, based upon their optimal transport formulation.
  This takes the form of a Moreau-Yoshida regularisation of the
  Wasserstein metric. 
  Solutions of the optimal transport formulation provide the shallow water layer
  depth represented as a measure, which is itself the push forward of an
  evolving measure under the semigeostrophic coordinate
  transformation. First, we propose and study an entropy
  regularised version of the rotating shallow water equations. Second, we
  discretise the regularised problem by replacing both measures with
  weighted sums of Dirac measures, and approximate the (squared)
  $L^2$ norm of the layer depth, which defines the potential energy.  We
  propose an iterative method to solve the discrete optimisation
  problem relating the two measures, and analyse its convergence. The
  iterative method is demonstrated numerically and applied to the
  solution of the time-dependent shallow water problem in numerical
  examples.
\end{abstract}

%\begin{keyword}
%\ccode{AMS Subject Classification: 
%\MSC 78A46 \sep 49Q22 \sep 65K10
%\end{keyword} 

%\setcounter{tocdepth}{2}
%\tableofcontents

\section{Introduction} 

The semigeostrophic (SG) approximation describes the large scale evolution of fluid flows in the limit of
the Rossby number (measuring relative size of the advection and Coriolis terms) going to zero in the distinguished limit where the 
Froude number (a ratio of transport to wave propagation timescales) is proportional to the square
of the Rossby number. Originally proposed to explain the formation and subsequent evolution of atmospheric fronts \citep{hoskins1971atmospheric}, the SG equations have 
been proposed more recently as a tool for understanding weather models in the SG limit \citep{cullen2007modelling,cullen2018use}. The SG
equations are advantageous since they do not explicitly support fast wave motions, and they can be solved numerically using optimal transport techniques that do not require
numerical dissipation for stability, even in the presence of fronts. These techniques were originally formulated via the "geometric algorithm" of \citet{cullen1993geometric}.
At the time, these calculations were limited by available computer power. Recently, there have been significant advances in computational methods driving efforts to revisit the optimal transport approach for the SG equations. The geometric algorithm 
approximates the source density as a weighted sum of Dirac masses and computes the optimal transport to a Lebesgue measure in the target domain (referred to as a ``semidiscrete" optimal transport problem).
This corresponds to the construction of Laguerre cells (also known as power diagrams) which have subsequently been well studied, i.e. \citet{merigot,kitagawa2019convergence} who suggest a damped Newton method. This semidiscrete method was used to modernise the geometric algorithm in \citet{BourneP,egan2022new,lavier2024semi}. Entropic regularisation is another popular approach for fully discrete problems where the source and the target measures are both weighted sums of Dirac masses. Here, through the addition of a Kullback-Leibler regularisation term, scaled by a small parameter $\epsilon$, is solved \emph{via} the dual problem and the Sinkhorn iterative procedure. The iterations alternate on the optimality conditions for the two dual potentials in the problem. The Sinkhorn iteration is well suited to fast implementation on GPUs \citet{cuturi2013sinkhorn},
and can be accelerated using scaling techniques \citep{Schmitzer,chizat2018scaling}. Further, the $\mathcal{O}(\epsilon \log(\epsilon))$ error due to regularisation can be efficiently corrected (``debiased") to $\mathcal{O}(\epsilon^2)$ \citep{Feydy}.
\cite{BCM24} applied this approach to the incompressible Boussinesq SG equations in the Eady vertical slice configuration. 

In this paper we present a fully discrete optimal transport approach to the rotating shallow water SG equations, as a stepping stone towards the compressible Euler SG equations.
The shallow water and the compressible Euler equations share a feature which is that the divergence-free condition, imposing incompressibilty of the flow, is replaced by a density that is transported by the flow according to the continuity equation. (In the case of the shallow water model, the ``density" is the volume of water per unit horizontal area which can change even though the fluid is assumed incompressible due to the vertical motion of the upper surface.) For the shallow water and Euler equations, the optimal transport problem associated to the SG approximation is generalised by removing the constraint on the target density and adding a potential energy term that penalises it instead \citep{shutts1987parcel}. The resulting formulation for the shallow water SG model was rigorously analysed in \citet{cullen2001variational}, 
and \citet{cullen2003fully} for the compressible Euler SG model. \citet{bourne2025semi} proposed and analysed an extension of the geometric algorithm to the compressible Euler SG model.
This paper is an adaptation to the entropic solution 
approach of the SG equation \citep{BCM24} to the SG shallow water equations. {However, entropy debiasing raises  new and interesting questions that are only tackled partially here}.
The optimal transport and entropic optimal transport tools  were described 
in \citep{BCM24} and we will refer the reader to this paper when needed.

The rest of this paper is structured as follows. Section 2 revisits the shallow water SG model derivation using OT tools. Section 3 presents the Entropic regularisation, debiasing, and solution algorithms.
Finally, numerical results and comments are given in Section 4.

\section{The OT formulation of the SWSG equations} 
\label{oft} 

\subsection{Shallow water (SW) equations} 

The shallow water equations are a standard model of geophysical fluid dynamics which can serve as a simplified model of the ocean or a layer in the atmosphere.
They are derived under the assumption of an inviscid, incompressible fluid in 3D with a free surface, in hydrostatic balance and assuming columnar motion.
Here, the fluid is defined on 2D planar geometry (in Cartesian coordinates $x_1,x_2$)  with Coriolis parameter $f$ and gravity parameter $g$ both assumed constant.
Then the shallow water equations for unknown velocity $U_t= (u_1,u_2)$  and height $h$, are given by
\begin{align}
\label{eq:3d B u}  \frac{Du_1}{Dt} - fu_2 & = - g \,  \frac{\partial h}{\partial x_1}, \\
\label{eq:3d B v}  \frac{Du_2}{Dt} + fu_1 & = - g\, \frac{\partial h}{\partial x_2}, \\
\label{eq:3d B h}  \frac{Dh }{Dt}  &  = -  h \, (  \frac{\partial u_1}{\partial x_1}  +\frac{\partial u_2}{\partial x_2}  ).
\end{align}
The first two equations are Newton's balance of forces and the third is the mass conservation expressed in terms of the water column height, \(h\). These equations are accompanied by
initial conditions for $U_0$ and $h$.
In this paper we consider the solution in a {\em convex} domain $\Omega \subset \R^2$
with  boundary $\partial\Omega$ with either  periodic or  rigid   boundary conditions $(u_1,u_2)\cdot
n=0$, where $n$ is the unit outward pointing normal to the boundary, or a mix of the two.

We recall the definition of the material derivative 
\begin{equation}
  \label{eq:D/Dt}
  \frac{D}{Dt} = \frac{\partial}{\partial t}
  + u_1\frac{\partial}{\partial x_1}
  + u_2\frac{\partial}{\partial x_2},
  \end{equation}
so that (\ref{eq:3d B h}) can also be written as a Eulerian conservative continuity equation,
\begin{align}
\label{ceh} 
\frac{\partial}{\partial t}  h + \frac{\partial}{\partial x_1} (h \, u_1) + \frac{\partial}{\partial x_2} (h \, u_2) = 0.
\end{align}
The total volume of water is conserved but the  height of water columns  varies according to the 
non divergence free velocity. We will model  $t \rightarrow h_t$ as a curve of densities over $\Omega$.

\subsection{SG SW (SGSW)  equations} 
\label{sgsw} 
The geostrophic approximation, which holds at large scales for slow moving flows,  neglects $Du_1/Dt$ and $Du_2/Dt$ in
(\ref{eq:3d B u}-\ref{eq:3d B v}), leading to the (divergence free) geostrophic velocity $U_{t,g} = (u_{1,g},u_{2,g})$
defined according to
\begin{equation}
  \label{eq:geobal}
  -f\, u_{2,g} = - g \, \frac{\partial h}{\partial x_1}, \quad
  f \, u_{1,g }= -\ g\, \frac{\partial h }{\partial x_2}.
\end{equation}
However, this is a purely diagnostic equation that does not predict
dynamics. In the {\em semi}geostrophic  approximation (see \cite{cullen2006mathematical} for a historical review) of (\ref{eq:3d B
  u}-\ref{eq:3d B v}) we neglect the acceleration of the ``ageostrophic'' part of the velocity $U_{ag} := U-U_g$,  and hence we replace $Du_1/Dt$, $Du_2/Dt$ by $Du_{1,g}/Dt$,
$Du_{2,g}/Dt$ (whilst retaining $u_1,u_2$ in \ref{eq:D/Dt}),
leading to
\begin{align}
\label{eq:3d sg u}  \frac{Du_{1,g}}{Dt} - fu_2 & = - g\, \frac{\partial h}{\partial x_1} {= -f u_{2,g}}, \\
\label{eq:3d sg v}  \frac{Du_{2,g}}{Dt} + fu_1 & = -  g \frac{\partial h }{\partial x_2} { = f u_{1,g}}, \\
\label{eq:3d sg h}  \frac{Dh }{Dt}  &  = -  h \, (  \frac{\partial u_1}{\partial x_1} +\frac{\partial u_2}{\partial x_2}  ).
\end{align}
To clarify how this equation might be solved, we use \ref{eq:geobal} to eliminate $u_{1,g}$, $u_{2,g}$  to obtain
\begin{align}
\label{eq:3d sgp u}  \frac{D}{Dt}\left(\frac{g}{f}\frac{\partial h}{\partial x_2} + f \, x_2\right) & =  g\, \frac{\partial h}{\partial x_1} , \\
\label{eq:3d sgp v}  \frac{D}{Dt}\left(\frac{g}{f}\frac{\partial h}{\partial x_1} + f \, x_1\right) & = -  g\, \frac{\partial h }{\partial x_2}    , 
\end{align}
whilst the last equation (\ref{eq:3d sg h}) is unchanged.

The semi-geostrophic  shallow water system is (\ref{eq:3d sg h}-\ref{eq:3d sgp v})  with unknowns 
 $(u_1,u_2,h )$  which are functions of $(t,x_1,x_2)$  (time and space). Note here that given $h$, ($u_1,u_2$) can be obtained from (\ref{eq:3d sgp u}-\ref{eq:3d sgp v}),
 and hence initial conditions are required for $h$ only. This reflects the SG approximation as a next order correction to the geostrophic balance condition
 that determines $u$ from $h$.
 
\subsection{Hoskins' transform and Cullen stability principle}

The optimal transport formulation  is  based on the change of  {\em physical coordinates} $X = (x_1 ,x_2 ) \in \Omega$ 
into  \emph{geostrophic coordinates} $Y = (y_1,y_2)$\footnote{Note that the geostrophic coordinates were denoted $G$ in \citep{BCM24}, we use $Y$ to avoid confusion with the gravity constant $g$.} , also known as Hoskins' transformation \citep{hoskins1975geostrophic},  
defined by
\begin{equation}
\label{cov}
  y_1 = x_1 +   \frac{g}{f^2}\frac{\partial h}{\partial x_1} , \quad  y_2 = x_2 +   \frac{g}{f^2}\frac{\partial h}{\partial x_2} .
  \end{equation}
The geostrophic domain  is the  image, deforming in time,  of the physical domain
under this map. We notice that it is a gradient 
\[
 Y   = \nabla P_t(X) ,
\]
where 
\begin{equation}
  \label{cov0}
  P_t (X)  = \frac{1}{2} \| X \|^2  +\frac{g}{f^2} \,  h_t(X). 
\end{equation}
Since $h$ depends on time, so does the domain in geostrophic coordinates $P_t$.

Under this change of variables,  (\ref{eq:3d sgp u}-\ref{eq:3d sgp v}) becomes
\begin{align}
\label{eg}\frac{D Y }{Dt} & =  f   \underbrace{\begin{pmatrix}
    0 & 1  \\
{-}1 & 0 
  \end{pmatrix}}_{=J}  ( Y -  X)  =   \, U_{t,g} (X)  , 
\end{align}
where the advective derivative (\ref{eq:D/Dt})  is still governed by the physical velocity $U_t = (u_1,u_2)$.\\

It is convenient from there to   switch to a  Lagrangian
description of the equations coupling particles in the physical and geostrophic domains. 
The Lagrangian description of fluid
dynamics is formulated in terms of a time dependent flow map
$\X_t:\Omega \to \Omega$ such that $\X_t(\X_0)$ describes the time
evolution of a moving fluid particle for each fixed $\X_0 \in
\Omega$.    Let us consider  particles  $\X_t$ moving  with velocity $U_t =(u_{1},u_2)$ and 
their images $ \Y_t(\Y_0)  = \nabla P_t (\X_t ( \X_0 ) )$ in geostrophic space.  
%Note that the change of variable $ X \rightarrow Y$ also depends on time  even though this is not explicit in the notations.
From (\ref{eq:D/Dt}-\ref{eg}) 
we obtain the following system of ODEs,
\begin{align} 
\label{Lx}
  \pp{}{t}\X_t & =  U_t (\X_t) , \\
\label{Ly}  
  \pp{}{t}\Y_t  & =  f\, J \cdot (  \Y_t -  \X_t)  = U_{t,g}  (\X_t ) .
\end{align} 
This is complemented by the Lagrangian form of (\ref{ceh}), which we reformulate now as follows. 
Let  $\mu = \L \mres \Omega$ where $\L$ is the Lebesgue measure on $\Omega$ and define, for  all time $t$,   $(h_t \mu) $ as  the {\em pushforward measure} of $ (h_0 \mu) $ by the  flow map $\X_0 \mapsto \X_t$
satisfying
$$
\int_B h_t d\mu  = \int_{\X_t^{-1}(B)} h_0 d\mu, 
\mbox{ for all $\mu$ measurable sets $B$. }
$$
 We will use the shorthand  notation for the pushforward,
 \begin{equation} 
\label{Lh}    (h_t\,\mu)   = (\X_t)_\# \, (h_0 \, \mu).
\end{equation} 
The Lagrangian system (\ref{Lx}-\ref{Lh})  is closed by (\ref{cov0})  applied to $(\X_t, \Y_t)$,
% is closed by Hoskins transformation relating \X \Y?
\begin{equation}
\label{cov2}
  \Y_t = \X_t +    \frac{g}{f^2}\nabla  h_t(\X_t),
  \end{equation}
as  $U_t$ is also an unknown.   \\

A dynamic solution in geostrophic coordinates alone,  can be  constructed, based on an 
additional assumption now known as the Cullen Stability Principle.
\begin{assumption}(Cullen Stability Principle)
\label{csp} 
The solution of (\ref{eq:3d sg u}-\ref{eq:3d sg h}) \ is such that 
 $P_t$ (depending on $h$  as in \eqref{cov0})  remains  a strictly convex potential for all times in $[0,T]$.  
 In other words, $h$ is $(-f/g)$-convex.
\end{assumption}
Then the  map $ X \rightarrow  \nabla P_t(X) $  is injective and  bijective onto $ \nabla P_t(\Omega)$. Its inverse is given by   $ Y \rightarrow \nabla Q_t(Y)  $ where $Q_t = P_t^*$ is the Legendre-Fenchel transform of $P_t$.  

If we are given $Q_t$, we can decouple 
the dynamics in geostrophic coordinates, according to 
\begin{align} 
\label{Ly2}  
  \pp{}{t}\Y_t  & =  f\, J \cdot (  \Y_t -  \nabla Q_t(\Y_t)) = 
   U_{t,g}  ( \nabla Q_t(\Y_t) ) ),
\end{align} 
the last equality being a consequence of (\ref{cov2}).
The computation of $Q_t$ independently of $\X_t$ is explained in the next section.

\begin{remark} 
We can also return to a Eulerian description of 
the flow in geostrophic coordinates. For an initial geostrophic distribution $\sigma_0$, the curve in time of measures 
$\sigma_t  = (\Y_t)_\# \sigma_0 $  is a distributional solution of 
\begin{align} 
\label{DY}  
  \pp{}{t} \sigma_t (Y) + \nabla \cdot ( \sigma_t (Y) \,  U_{t,g}  ( \nabla Q_t(Y) ) = 0.
\end{align} 
\end{remark}

\subsection{Optimal Transport Formulation}

The following result links the  flows in 
physical and geostrophic coordinates, showing that 
the dynamics in physical coordinates can be 
recovered from the geostrophic one assuming the maps 
$\{\nabla Q_t\}$ are given.
\begin{proposition} \label{EQI} 
%We assume (CSP) holds and the map $\{\nabla Q_t\}$ are smooth and given for $t\in [0,T]$. 
Let $P_0$ and $h_0$  be an initial convex potential and height defined over $\Omega$ and define $ \sigma_0 =  (\nabla P_0)_\# \, (h_0 \, \mu) $. 
Additionally, consider the family of maps $\{ Y_0 \rightarrow \Y_t\}_t $ (a solution of (\ref{Ly}), for instance) and 
 a  family of maps   $\{\nabla Q_t\}$ defined on 
the support of $\sigma_t$.

Defining 
\[ \begin{array}{c} 
\sigma_t =  (\Y_t)_\# \sigma_0 \\[8pt]
\X_0 \rightarrow \X_t(\X_0) := \nabla Q_t \circ \Y_t \circ \nabla P_0 (\X_0), 
\end{array} 
\] 
and assuming that
\begin{equation}
\label{A1} 
(h_t   \mu)  = (\nabla Q_t )_\#  \sigma_t,
\end{equation}
implies that
 \[
     (h_t   \mu ) =  (\X_t)_\#  (h_0  \mu) .
\]
\end{proposition} 
\begin{proof} 
The result follows from the decomposition of pushforward for composition of maps, as follows.
\[
\begin{tikzcd}
\Y_t(\Y_0) \arrow[r, " \nabla Q_t "] &   \X_t(\X_0)  \\
\Y_0 \arrow[u, " t " ]     & \X_0 \arrow[u, " t "']     \arrow[l, " \nabla P_0 " ]    \end{tikzcd} 
\quad \quad 
\begin{tikzcd}
 \sigma_t \arrow[r, " (\nabla Q_t)_\# "] &   h_t \, \mu    \\
\sigma_0 \arrow[u, " (\Y_t)_\# " ]     &    h_0 \,  \mu  \arrow[u, " (\X_t)_\# "']     \arrow[l, " (\nabla P_0)_\# " ]    \end{tikzcd}
\]
\end{proof} 

%Proposition \ref{EQI} will be used to  decouple equation (\ref{Ly}) from the system. 
We  now  need to define 
the family of  maps $(\nabla Q_t)_t$ and heights $(h_t)_t$ 
that are compatible with 
assumptions  (\ref{cov0}-\ref{A1}) and (CSP).  This was obtained based on energy minimisation  considerations  in \citet{cullen1989properties, cullen2001variational,cullen2006mathematical}. It is also contained in  
 the more recent ``unbalanced"   optimal Transport problem
\citep{chizat2017unbalanced} or the concept of 
 Moreau-Yoshida envelope in Wasserstein space \citep{sarrazin2022lagrangian}.
This is summarised in the following proposition.
\begin{proposition}[] 
\label{msr} 
Given a measure $\sigma$ and $\mu$ the volume measure on $\Omega $ defined as above, we consider the minimisation problem  
\begin{align} 
 \label{e3} &  \inf_{  
 h  \in \P(\Omega) }   \Ec_{\sigma} (h) ,
  \end{align}
  where 
  \begin{equation}
\label{energies} 
\begin{array}{l}
\Ec_\si = \Kc_\sigma + \Pc, \\[10pt]
\Kc_{\sigma} (h) =   \ f^2     \, \OT(h \, \mu, \sigma)
:= f^2\inf_{G,\sigma=G\#h\mu} \frac{1}{2}\int_\Omega \|G(X)-X\|^2
h\mu(\diff X), 
   \\[10pt]
   \Pc (h) =\dfrac{ g}{2}  \,  \int_\Omega  \| h  (X)  \|^ 2 \, d\mu(X), \mbox{  if $h  \in \L^2_{\mu}(\Omega)$ and $+\infty$ otherwise, }
\end{array}
\end{equation}
are the total, kinetic and potential  energies respectively, and  $\OT$ is the usual 2-Wasserstein distance squared (see Proposition 2 in \citep{BCM24}). 
The following properties hold.
 \begin{itemize} 
\item[(i)] (\ref{e3}) has a unique solution and also   admits the dual formulation,
\begin{equation}
  \label{dual2} 
  \inf_{  
 h  \in  \P(\Omega) }   \Ec_{\sigma} (h)   = \sup_{
    \begin{array}{c} \u \in \Cc(\Omega), \v \in \Cc(\R^d)  \mbox{ s.t. } \\ 
      \u(X) +  \v(Y) \le  {\frac{1}{2}} \,  \| Y-X \|^2, \, \forall (X,Y)  \end{array}
  }
 \Dc_\sigma(\u,\v),
\end{equation} 
where
$$
\Dc_\sigma(\u,\v) := \int_{} \v  \, \sigma( \diff Y)  -  \frac{f^2}{2g} \int_\Omega  \| \u (X)  \|^ 2 \, d\mu(X).   
$$
%and 
%$ \dfrac{g}{f^2} h   =  -\u   $  (assuming $\nabla \u$ does not vanish 
%\changes{ see Sarrazin for discussions ...}).
\item[(ii)] Strong Fenchel Rockafellar duality holds 
 and there exist unique 
minimisers/maximisers satisfying 
\begin{equation}
  \label{danskins} 
  h  = -\partial_{\u} \Dc_\sigma(\u,\v) = - \frac{f^2}{g}  \u ,
   \quad \quad   \v = \partial_{\sigma} \Ec_{\sigma} (h).
  % = f^2    \partial_{\sigma}      \, \OT(h \, \mu, \sigma_t) 
\end{equation}

\item[(iii)] The optimal $(\u,\v)$  are  the Kantorovich potentials for the optimal transport problem between $(h\mu)$ and $\sigma$.  The 
associated ``Brenier'' potentials,
\[ \mbox{ $ P(X) = \dfrac{\| X \|^2}{2} - \u (X)  $  and  $Q(Y) =  \dfrac{\| Y \|^2}{2} - \v (Y) $},\]
satisfy 
\[
\begin{array}{c}
  \mbox{$P$ is convex,  $Q = P^*$, $ (h\mu) = (\nabla Q )_\#  \sigma$ and    $ \sigma  = (\nabla P )_\# (h\mu) $  }. 
\end{array}
\]
   \end{itemize} 
\end{proposition} 
We are now ready  to formulate the  SWSG equation in geostrophic space.
\begin{theorem}[OT SWSG] 
\label{pot} 
Assume $(h_0, \sigma_0)$  and $\mu$ are given as per Proposition \ref{EQI}. 
Further, assume that we are given  $ \{ \Y_t, h_t\} $   solutions   for all time $t \in [0,T]$  of the following system,
\begin{equation}
\label{GSys}    
\left\{ 
\begin{array}{l}
\pp{}{t}\Y_t  = f\, J \cdot  (  \Y_t  -  \nabla Q_t(\Y_t )  ) 
\quad  \Y_0 = \Id,  \quad  \sigma_t =  (\Y_t)_\# \sigma_0 ,   \\[10pt]

(\u_t,\v_t) = arg \sup_{
    \begin{array}{c} \u \in \Cc(\Omega), \v \in \Cc(\R^d)  \mbox{ s.t. } \\ 
      \u(X) +  \v(Y) \le  {\frac{1}{2}} \,  \| Y-X \|^2, \, \forall (X,Y)  \end{array}
  }
 \Dc_{\sigma_t}(\u,\v),\\[10pt]
  h_t  :=  - \dfrac{f^2}{g}  \u_t   , \quad 
 Q_t (Y)  =  \frac{\|Y\|^2}{2} - \v_t(Y)   .
\end{array}
\right.
\end{equation}

Then, for $\X_t := \nabla Q_t ( \Y_t)$, the following properties hold.
\begin{itemize} 
\item[$i)$] $ \{  \X_t  , Y_t , h_t \} $ are solutions of the full Lagrangian SWSG system  (\ref{Lx}-\ref{cov2}).
\item[$ii)$] 
%The SG velocity 
%is $ U_{g,t}(\X_t)$ is $\pp{}{t}\Y_t $ applied at   
%$\X_t =  \nabla Q_t(\Y_t )  $  or 
$ U_{g,t}(\X_t) = f\, J \cdot (  \nabla P_t(\Y_t) -\X_t) $
where $ P_t = Q_t$.
\item[$iii)$] The geostrophic flow equation can be written in 
compact form using the Wasserstein $\sigma$ gradient, according to
\begin{equation}
\label{FlP} 
\pp{}{t}\Y_t  = f\, J \cdot  \nabla   \partial_{\sigma} \Ec_{\sigma_t} (h)  = f\, J \cdot  \nabla \v_t(\Y_t) .
\end{equation} 

 \end{itemize} 

\end{theorem} 
\begin{proof}
We apply
Propositions \ref{EQI} and \ref{msr} in sequence. In particular, (\ref{cov2}) is obtained by taking the gradient 
of $P$ (which is convex and therefore a.e-differentiable) defined in 
(iii) of Proposition \ref{msr}, and using the characterisation of $h$ in $(ii)$ of Proposition \ref{msr}. Point $iii)$ is likewise established by taking the 
gradient of $Q$ defined in ($iii)$ of
Proposition \ref{msr}.  
\end{proof} 
\begin{remark}
\label{velo}
The full velocity $ U_t(\X_t)$ (rather than the geostrophic velocity $U_g$) is not needed  to solve (\ref{FlP}), but can be recovered from the solution as $\pp{}{t}\X_t $ as
a postprocessing diagnostic if necessary.
\end{remark}
\begin{remark}
Combining 
Proposition \ref{msr} $(iii)$ with Proposition \ref{EQI} $(ii)$, (\ref{DY}) can be written as 
\begin{align} 
\label{DY2}  
  \pp{}{t} \sigma_t  + \nabla \cdot ( \sigma_t  \, J\cdot  \nabla \, \partial_{\sigma} \Ec_{\sigma_t} ) = 0 
\end{align} 
which satisfies the Wasserstein Hamiltonian system definition of 
\citet{ambrosio2008hamiltonian}. Consequently, the total energy is conserved. Further, the vector field $J\cdot  \nabla \, \partial_{\sigma_t} \Ec_{\sigma_t}$ is divergence free. Then, provided that $Y_t$ remains smooth and invertible, $\pp{}{t} \sigma_t(\Y_t )  = 0 $.
\end{remark}

\section{Entropic 
SWSG approximation and Sinkhorn algorithm} 
%regularized Energy and Sinkhorn adaptation}

In this section, we adapt  the entropic solution 
approach to the incompressible SG equation of \citep{BCM24} to the shallow water SG equation. 
This consists of one simple change: we add an entropic regularisation term to the 
optimal transport problem (\ref{dual2}). It convexifies the problem and enforces the inequality constraint on the potentials in a soft way. A further approximation is required to make a computational method, namely the grid approximation of the potential energy, but we leave this aspect until later. The convergence analysis  of the 
time discretised and entropy regularised problem (as in \cite{carliermalamut24})  is also left for further studies.
In this section, we focus on the solution  of (\ref{OC0}) (and its Sinkhorn Divergence debiased version) using 
the Sinkhorn Algorithm. 

\subsection{Entropic regularisation}
Starting from the primal formulation \eqref{e3}, we replace $\OT(.,.)$ in the kinetic energy with $\OT_\ep(.,.)$ (as defined in section 3.2 in \citet{BCM24}),
for $0< \ep <<1$.
Passing to the dual formulation, \eqref{dual2} is replaced by
the entropy regularised dual  problem,
\begin{equation}
  \label{dualreg} 
  \begin{array}{ll} 
\displaystyle \sup_{ \u \in \Cc(\Omega), \v \in \Cc(\R^d) }  \Dc_{\si,\ep}(\u,\v) :=
  & \displaystyle \int_{\R^d} \v (Y)   \, d\sigma(Y)  -  \dfrac{f^2}{2\,g} \int_\Omega  \| \u (X)  \|^ 2 \, d\mu(X)  -   
 \\[12pt]  & 
  \ep \, \displaystyle \int_{\Omega \times \R^d } 
   \left( e^{(\u(X) +  \v(Y))/\ep}     \,  \Ke(X,Y) -1\right)  \,  d\mu(X) \,d\sigma(Y),
 \end{array}   
\end{equation} 
where we define the heat kernel
\begin{equation}
\label{HK} 
\Ke(X,Y) = e^{-\|X-Y\|^2/\epsilon}.
\end{equation}
The extra third term is strictly convex, ensuring existence and uniqueness.
It also enforces the positivity constraint in (\ref{dual2}) in a soft way.
The solution of the regularised problem \eqref{dualreg} is characterised as follows.

\begin{proposition}
    (\ref{dualreg}) has a unique solution $(\ue, \ve)$  with the following properties.

\begin{itemize}
    \item[(i)]  The optimality conditions for the optimiser are 
 \begin{equation}
 \label{OC0}
 \left\{ 
 \begin{array}{l}
\ve(Y) = -\ep  \, \log \left(  \displaystyle \int_{\Omega \times \R^d }   
   (  e^{\ue(X)/\ep} \,\Ke(X,Y) \,   d\mu(X)  \right), \quad \forall Y, \\[10pt]
\ue(X) = -\ep  \, W_0 \left(  \dfrac{g}{\ep f^2} \displaystyle \int_{\Omega \times \R^d }   
     e^{\ve(Y)/\ep} \,\Ke(X,Y)     \,   d\sigma(Y)  \right)  , \quad \forall X,
 \end{array}
 \right.
 \end{equation}
 where $W_0$ is the $0^{\mbox{th}}$ branch of the Lambert function.

 \item[(ii)] We define the ``entropic'' water height $h_\ep:=- \frac{f^2}{g}  \ue$ .  
 Then   $(\ue, \ve)$ solve
 \begin{equation}
  %\label{dualreg} 
  \begin{array}{ll} 
 \OT_\ep(\sigma , \he \mu) := &\displaystyle \sup_{ \ue \in \Cc(\Omega), \ve \in \Cc(\R^d) } 
 \displaystyle   \int_{} \ve (Y)   \, d\sigma(Y)  + 
  
   \int_\Omega  \ue (X) \,\he(X)  \, d\mu(X)  -   
 \\[12pt]  & 
  \ep \,  \int_{\Omega \times \R^d } 
   ( e^{(\ue(X) +  \ve(Y))/\ep}     \,  \Ke(X,Y) -1)  \,  d\mu(X) \,d\sigma(Y),
 \end{array}   
\end{equation} 
 and $h_\ep$ 
 is the solution of the primal problem 
\begin{equation}
  \label{primalreg} 
\displaystyle \min_h 
 \Ec_{\si,\ep}(h) :=  {\OT_\ep(\sigma,h) + \frac{g}{2f^2} \int \|h \|^2 d\mu}.
 \end{equation} 
 \item[(iii)] The gradient of $\ve$ is also the barycentric map, given by
 \begin{equation}
\label{BDC}
\nabla  \ve(Y) =  \Y-\dfrac{1}{\sigma(Y)}\int_{\Omega \times \R^d }  X \, 
    e^{  (\ue(X) +  \ve(Y))/\ep}  \,\Ke(X,Y)   \,   d\mu(X).
\end{equation}
\end{itemize}
\end{proposition}

\begin{proof}
Property $(i)$ follows directly from the optimality conditions of the 
the strictly concave maximisation (\ref{dualreg}).
For $(ii)$, it is sufficient to plug the 
Legendre-Fenchel dual, 
\begin{align}
&  \frac{1}{2} \| \ue\|^2  = \sup_{h_\ep}
 h_\ep  \ue - \frac{1}{2} \| h_\ep\|^2,
\end{align}
into \ref{dualreg}.
We take the gradient in the first equation of (\ref{OC0}) to get $(iii)$.
\end{proof}
We now gather a few results on the convergence as 
$\ep \rightarrow 0$ of these quantities.

\begin{proposition}[] 
\label{Emsr} 
As $\ep \rightarrow 0$, we have the following properties of convergence to the solution
of the unregularised problem
(which also apply when the measures $\mu$ and $\sigma$ are discrete).
 \begin{itemize} 
\item[(i)] Convergence of the value: Let $( \u_\ep,\v_\ep)$ solve the regularised problem with parameter $\ep$,
and let $(\u, \v)$ solve the unregularised problem. Then
$0 \leq \Dc_\sigma( \u_\ep,\v_\ep) - \Dc_\sigma( \u,\v) \leq C \ep\log(1/\ep)$, for some constant $C$ depending on $\sigma$.

\item[(ii)]
$  \|h_\ep  - h\|_{L^2(\mu)} \leq C' \sqrt{\ep\log(1/\ep)}$ for some constant $C'$ depending on $\sigma$.

\item[(iii)] If the Brenier map from $\sigma$ to $h$ is globally Lipschitz, then $\nabla  \psi_\ep \to \nabla  \psi$ in $L^2(\sigma)$.
\end{itemize} 
\end{proposition}
\begin{proof}
  \begin{itemize}
\item[(i)] By duality, we have that 
\begin{equation}
\label{opti}
    \Dc_\sigma(\phi,\psi) = \min_h \Ec_{\sigma}(h) \quad \text{and} \quad \Dc_\sigma(\phi_\ep,\psi_\ep) = \min_h \Ec_{\sigma,\ep}(h).
\end{equation}

Since the entropic optimal transport cost $\OT_\ep$ is increasing in $\ep$, it is clear that $\Ec_{\sigma,\ep}$ in \eqref{primalreg} is also increasing in $\ep$.
Thus, the minima in $h$ described in \eqref{opti} are also ordered according to 
\[\Dc_\sigma(\phi,\psi) \leq \Dc_\sigma(\phi_\ep,\psi_\ep). \] 
For the second inequality, choose $h$ as a candidate in the minimisation of $\Ec_{\sigma,\ep}$. Then, equation \eqref{opti} gives $\Dc_\sigma(\phi_\ep,\psi_\ep) \leq \Ec_{\sigma,\ep}(h)$. So,
\[
\Dc_\sigma(\phi_\ep,\psi_\ep) - \Dc_\sigma(\phi,\psi) \leq \Ec_{\sigma,\ep}(h) - \Ec_{\sigma}(h) = \OT_\ep(\sigma,h) - \OT_0(\sigma,h),
\]
and it is known that this suboptimality converges in $\ep\log(1/\ep)$ (see \citet{carlier2017convergence}). \\
\item[(ii)] The function $\Ec_{\sigma}$ is $\frac{g}{f^2}$-strongly convex, with minimum at $h$, so 
\[\left\|h - h_\ep \right\|_{L^2(\mu)}^2\leq \frac{2f^2}{g}\left(\Ec_{\sigma}(h_\ep) - \Ec_{\sigma}(h)\right).\]
This left hand side is equal to $\Dc_\sigma(\phi_\ep,\psi_\ep) - \Dc_\sigma(\phi,\psi)$, which is dominated by $C\ep\log(1/\ep)$ as seen above. So 
\[ \left\|h - h_\ep \right\|_{L^2(\mu)} \leq C'\sqrt{\ep\log(1/\ep)}.
\]
\item[(iii)] First, note that, as a result of the relation between the Kantorovich potential $\psi$ and the Brenier potential $Q$, $\nabla \psi(Y) = Y -\nabla Q(Y) $. Combining this fact with relation \eqref{BDC} yields
\[
\nabla\psi_\ep(Y) -\nabla\psi(Y) = \nabla Q(Y) - \dfrac{1}{\sigma(Y)}\int_{\Omega \times \R^d }  X \, 
    e^{  (\ue(X) +  \ve(Y))/\ep}  \,\Ke(X,Y)   \,   d\mu(X).
\]
By Jensen's inequality,
\[ \| \nabla \psi_\ep - \nabla \psi \|_{L^2(\sigma)}^2 \leq \int | X - \nabla Q(Y) |^2 d \gamma_\ep(X,Y),
\]
where $d\gamma_\ep(X,Y) := e^{  (\ue(X) +  \ve(Y))/\ep}  \,\Ke(X,Y)   \,  d\mu(X)d\sigma(Y)$ is the entropic optimal transport plan between $h_\ep$ and $\sigma$ (see \cite{PeyreBook} for more information). The measure $\gamma_\ep$ has support on the product space $\Omega \times \R^d$ and marginals $h_\ep,\sigma$.

Arguing as  in \cite{berman2020stability} and  \cite{liNochetto2020stability}, who build upon an earlier argument of \cite{gigli2011mintytrick}, we can use the following inequality,
\[
\frac{| X - \nabla Q(Y) |^2}{2L} \leq \left( \frac{|X-Y|^2}{2} - \phi(X) - \psi(Y) \right),
\]
where $L$ is the Lipschitz constant of the Brenier map, so that
\begin{align} \nonumber
  \frac{\| \nabla \psi_\ep - \nabla \psi \|_{L^2(\sigma)}^2}{2L} \leq & \int \frac{|X-Y|^2}{2}d\gamma_\ep(X,Y) \\
  & \qquad - \underbrace{\int \phi(X) d\gamma_\ep(X,Y)}_{\int \phi(X) h_\ep(X) d\mu(X)} -\underbrace{\int \psi(Y) d\gamma_\ep(X,Y)}_{\int \psi(Y) d\sigma(Y)}.
  \label{eq:Lino}
\end{align}
The first term converges to $OT_0(\sigma,h)$. Since $\phi = -\frac{2g}{f^2}h$ is bounded in $L^2(\mu)$ and $h_\ep$ converges to $h$ in $L^2(\mu)$, then
\[
\int \phi h_\ep d\mu  \to \int \phi h d\mu,
\]
by the Cauchy–Schwarz inequality.
By Kantorovich duality,
\begin{align}
  \nonumber
  \int (\phi h_\ep)(X) d\mu(X) + \int \psi(Y) d\sigma(Y) &\to \int (\phi h)(X) d\mu(X) + \int \psi(Y) d\sigma(Y) \\
  & = OT_0(\sigma,h).
\end{align}
Hence, the right-hand side of equation \eqref{eq:Lino} tends to 0, and so does the left-hand side.
\end{itemize}
\end{proof}

The computation of $(\ue,\ve)$  relies 
on an iterative relaxation of (\ref{OC0}), in the manner of Sinkhorn.
In our setting this amounts to the following.
\begin{equation}
 \label{SINK}
 \left\{ 
 \begin{array}{l}
\ve^{k+1}(Y) = -\ep  \, \log \left(   \displaystyle \int_{\Omega \times \R^d }   
    e^{  \ue^k(X)  /\ep}  \,\Ke(X,Y)     \,   d\mu(X)  \right), \quad \forall Y, \\[10pt]
\ue^{k+1}(X) = -\ep  \, W_0 \left(  \dfrac{g}{\ep f^2}  \displaystyle \int_{}   
   e^{  \ve^{k+1}(Y)  /\ep}   \,\Ke(X,Y)    \,   d\sigma(Y)  \right)  , \quad \forall X.
 \end{array}
 \right.
 \end{equation}
The convergence  proof of this variant of Sinkhorn is available in a 
more general setting in \citet{chizat2018scaling}.
The proof in \citet{dimarinogerolin} can also be 
adapted. We note that replacing the $\log$  by the Lambert 
function in the second equation actually improves the contraction rate of the method. 

\begin{remark}
From a computational perspective, it is important to notice that the formula \eqref{BDC}, giving both the transport map and flow speed, does not involve derivatives, providing a natural extension of the gradient when $\mu$ is a discrete probability  measure.
\end{remark}

\subsection{Debiasing $\OT_\ep$ with  with Sinkhorn divergence }

The Entropic regularisation introduces an $\ep$ dependent bias in 
the kinetic energy $\Ec_{K,\sigma} = f^2 \,\OT $ and associated Wasserstein gradient 
$\nabla  \partial _\sigma \Ec_{K,\sigma} $ (this is (\ref{BDC})). 
A debiased version, referred to as ``Sinkhorn divergence'', has been 
proposed \citep{genevay,Feydy,chizatfaster,pooladian22}, giving the following 
correction of $\OT_\ep$,
\begin{equation}
 \label{SD}
 \Sc_\ep( h\,\mu, \sigma) =  \OT_\ep( h\,\mu, \sigma) - \dfrac{1}{2} \left( 
 \OT_\ep( h\,\mu, h\, \mu) + \OT_\ep( \sigma, \sigma) \right). 
 \end{equation}
 %Clearly $\Sc_\ep(h\,\mu, h\,\mu)=0$. 
 Further, Theorem 1 of \citet{Feydy} shows that $\Sc_\ep$ is positive and convex in its two variables and continuous w.r.t the Wasserstein topology. It is also built to recover the natural identity 
$\Sc_\ep(\nu,\nu)=0$, which holds for $\OT$ but does not hold for $\OT_\ep$.
Based on the small $\ep$ asymptotic expansion of 
$\OT_\ep$ \citep{carlier2017convergence,conforti2019formula,pal2019difference},
$\Sc_\ep$  is an approximation of $\OT$ of order $O(\ep^2)$  compared 
with $O(\ep\,\log{\ep})$ for $\OT_\ep$ \citep{Feydy,chizatfaster}.
This has been established rigorously in the continuous setting under technical assumptions on the marginal measures, such as compact support.
Given both marginal measures, $\Sc_\ep$ makes a better proxy of the unregularised 
Wasserstein distance and is still easy to compute, with just three independent $\OT_\ep$ problems to solve and combine. 

Going further, the debiased entropic map, given by  $\sigma$ Wasserstein gradient   $\nabla \partial_\sigma \{ \OT_\ep( h\,\mu, \sigma) -  \OT_\ep( \sigma, \sigma)/2 \}$, is computable at the same cost.  In our setting, this means replacing
 $ \nabla \v$ by 
$ \nabla  \left( \ve -  \ve^S \right)$ where $\ve^S$ is the Kantorovich potential 
solution of the dual formulation of the symmetric $\OT_\ep( \sigma, \sigma)$ problem.    As $\sigma$ is given, this is easily 
and independently 
computed  by the  Sinkhorn iteration,
\begin{equation}
\label{PSIS}
\ve^{S,k+1}(Y) = -\ep \log\left( 
\SI{ e^{(  \ve^{S,k}(Y'))/\ep} \,\Ke(Y',Y)  }{ d\sigma(Y') }
\right),
\end{equation} 
followed by applying the barycentric map formula \eqref{BDC} to the resulting potential $\ve^{S,\star}$ at convergence $k\to \infty$,
 \begin{equation}
\label{BDC2}
\nabla  \ve^{S,\star}(Y) =  Y - \dfrac{1}{\sigma(Y)}\int_{\Omega \times \R^d }  X \, 
    e^{  (\ve^{S,\star}(X) +  \ve^{S,\star}(Y))/\ep}  \,\Ke(X,Y)   \,   d\mu(X).
\end{equation}

Based again on the asymptotic characterisation in the continuous setting of \cite{conforti2019formula}, \cite{pooladian22} show that the correction does not degrade the accuracy 
of the approximation for smooth continuous transport maps. The bias remains of order $O(\ep)$ in $L^2$ norm.
They also show that discretisation by a weighted sum of Dirac masses may degrade the correction. 
Nevertheless, our experiments (see later sections) with our discretisation of 
the SWSG shows that debiasing 
improves the solution, at least for the water height.
For samplings of compactly supported marginal measure in particular, the Sinkhorn divergence map corrects at least the entropic regularisation induced contraction at the boundary \citep{Feydy},
which we otherwise observed to be significant in our experiments.

\subsection{Debiasing the Entropic SWSG problem} 

The Entropic approximation of SWSG consists of replacing $\OT$ by $\OT_\ep$ in (\ref{e3}-\ref{energies}) and then $\nabla \v_t $ by $\nabla \v_{t,\ep}$
(pointwise in $t$) in \ref{FlP}.  
Debiasing the Entropic SWSG problem consists of replacing $\OT_\ep$ by $\Sc_\ep$  and then  $\nabla \v_{t,\ep}$ by $ \nabla  \left( \v_{t,\ep} - \v^S_{t,\ep} \right) $.  
Since we minimise over $\he$ (\ref{e3}), the debiasing symmetric part $ \OT_\ep( \he\,\mu, \he\, \mu)$  is coupled to the entropic kinetic energy $\OT_\ep( \sigma, \he\, \mu) $ and the potential energy $\Pc(\he)$.  Thus, unlike the other symmetric part $ \OT_\ep( \sigma, \sigma)$, it cannot be computed independently. 
Below, we develop iterative methods for this new problem. \\

First, we recall that  while the  convexity of $\OT_\ep( \he\,\mu, \sigma) $  in $\he$ is well known, the less obvious  convexity of the negative symmetric terms has been established 
 using the  change of variable $ \us = e^{{\ue^S}/{\ep}} $ in the dual formulation ($\ue^S$ being the classical Kantorovich potential),
 \begin{align}
 \label{CVC}
 -\OT_\ep( \he\,\mu,  \he\,\mu) =   \ep \min_{ \us \in \Cc^+(\Omega)}   %\SP{a,b,c}
& - \SI{\he(X)\, \log(\us(X))}{d\mu(X)} - \frac{1}{2}  +  \\ 
& \frac{1}{2}  \SI{\us(X) \, \us(X') \,  \,\Ke(X,X') }{d\mu(X) \,d\mu(X')} .
 \nonumber \end{align}
Propositions 3 and 4 of \cite{Feydy} assure the positivity of the optimal $\us$.
Therefore, we can  use (\ref{CVC}) inside the Sinkhorn divergence (\ref{SD}) and remove the constant symmetric $\sigma$ part, before plugging the result  in place of  $\OT$ in (\ref{energies}). 
 We arrive at the convex/concave saddle point problem,
\begin{align}
 \label{SPP}
 \inf_{\he\in \P(\Omega), \us \in \Cc^+(\Omega)}   \, \sup_{ \ue,\ve}
&  \quad F(\he,\us,\ue,\ve),
\end{align}
where
\begin{align}
F(\he,\us,\ue,\ve) := & \SI{\ve(Y)}{d\sigma(Y)} + \SI{\ue(X) \, h(X) }{d\mu(X)} \, +  \nonumber \\[8pt]   & 
- \ep 
\SI{\he(X)\,  \log(\us(X))}{d\mu(X)}  \, -  \nonumber  \\[8pt]   &
\ep \,  \SI{  \left( e^{(\ue(X) +  \ve(Y))/\ep}  \,\Ke(X,Y)-1 \right)  }{ d\mu(X) \, d\sigma(Y) }
    \,  +  \nonumber \\[8pt]    &
\dfrac{\ep}{2}\, \SI{\us(X) \, \us(X') \,  \,\Ke(X,X') }{d\mu(X) \,d\mu(X')}
 \,  +  \nonumber \\[8pt]      &
 \dfrac{ g}{2} \, \SI{  \| \he  (X)  \|^ 2}{d\mu(X)},
  \nonumber
 \end{align}
where the Coriolis parameter is  set to $f=1$  and we have added/removed constants to ease the computations.

 Existence and uniqueness of a saddle point  solution to the system 
 \begin{equation}
 \label{Saddle}
   \nabla_{\he,\us,\ue,\ve} \, F(\he,\us,\ue,\ve) =0   
 \end{equation}
 follows from the strict convexity/concavity of $F$ (see  
 \cite{ekeland1976convex} chap.6 for example)  as well as the convergence of  steepest ascent-descent gradient methods of the form
 \begin{align}
& \label{AD1} 
(\he^{k+1},\us^{k+1})  = (\he^{k},\us^{k}) - t \, \nabla_{\he,\us} \, F(\he^k,\us^k,\ue^k,\ve^k),  \\[8pt]
& \label{AD2} 
(\ue^{k+1},\ve^{k+1})  = (\ue^{k},\ve^{k}) + t \, \nabla_{\ue,\ve} \, F(\he^k,\us^k,\ue^k,\ve^k), 
 \end{align}
 with  relaxation parameter $t$ chosen such that  
 $0< t < \frac{\lambda^2}{2\, L} $ 
 (in finite dimension, $L$ the Lipschitz constant of $F$ and $\lambda$ bounding below the modulus of convexity/concavity). There exist many refinements of this method in the literature (See  
\cite{Mokhtari20} and references therein).
 
\subsection{A Sinkhorn approach for the solution of (\ref{Saddle}) }

We rewrite (\ref{Saddle}) explicitly as the system 
\begin{align}
\label{DPSI}
& 1 = \SI{ e^{(\ue(X) +  \ve(Y))/\ep} \,\Ke(X,Y)}{d\mu(X) ,} \\[8pt] 
\label{DPHI}
& h(X) = \SI{ \left( e^{(\ue(X) +  \ve(Y))/\ep}   \,\Ke(X,Y) \right)}{d\sigma(Y) } \\[8pt]
& u(X) =  e^{ \frac{\ue(X)  + g\, h(X)}{\ep}  }  , \label{DH}  \\[8pt]
& h(X)  =  u(X) \, \SI{ u(X') \,  \,\Ke(X,X') }{d\mu(X')}  . \label{DU}
\end{align}

Since (\ref{DH}-\ref{DU}) arise  from a strictly convex problem, ithe system has a unique solution  $(\he[\ue],\us[\ue])$ depending on $\ue$, but it cannot be solved explicitly. Formally, eliminating $(\he,\us)$ yields
\begin{align}
 \label{DUALSreg}
 \sup_{ \ue,\ve}
& \SI{\ve(Y)}{d\sigma(Y)}  +\J_h(\ue) -
\ep \,  \SI{ \left( e^{(\ue(X) +  \ve(Y))/\ep} \,\Ke(X,Y) -1 \right) } {   d\mu(X) \, d\sigma(Y) },
 \end{align}
and
\begin{equation}
    \label{EH} 
    \J_h(\ue) = -  \dfrac{ g}{2} \, \SI{  \| h[\ue] (X)  \|^ 2}{d\mu(X)} \,-  \,\dfrac{\ep}{2} \,\SI{ h[\ue](X)}{ d\mu(X)}. 
\end{equation}
We do not have proof of 
concavity in $\ue$ of this new problem.
The convergence 
of a coordinatewise ascent in $(\ue,\ve)$  algorithm similar to
 Sinkhorn is unclear and is left for further research. \\

Instead, we develop a heuristic 
iterative relaxation method ``\`a la Sinkhorn''  to solve 
the optimality system (\ref{DPSI}-\ref{DU}), and provide experimental 
evidence of convergence for it.
We notice that $-\ue$ and $g\,\he$ are comparable (and equal when $\ep =0$), whilst $\us$ depends on the symmetric potential.
 It  is therefore tempting to eliminate $\he $ using (\ref{DH}) in (\ref{DPSI}-\ref{DU}). 
We find (taking the $\log$ when necessary),
\begin{align}
\label{DP2} 
& \ve(Y) = -\ep \, \log\left(  
\SI{ 
e^{ \ue  /\ep} \,\Ke(X,Y)}{d\mu(X) }
\right)  \, ,  \\[6pt]
\label{DF2} 
&   \ue(X) = - \ep \, W_0 \left( \frac{g}{\ep} \, u(X) \, 
 \SI{  e^{   \ve(Y))/\ep} } {  \Ke(X,Y) \,  d\sigma(Y) }
\right)    \, ,    \\[8pt]
& g\, u(X)  = \dfrac{ \ep \, \log(u(X)) -\ue(X)}{ \SI{ u(X') \,  \,\Ke(X,X') }{d\mu(X')} } . \label{DU2}
\end{align}

If $u=1$ then (\ref{DP2},\ref{DF2}) is exactly (\ref{OC0}), so we suggest to try and adapt Sinkhorn. This leads us to the following 
iterative algorithm,
\begin{align}
\label{DP2S} 
& \vekp(Y) = -\ep \, \log\left(  
\SI{  e^{ \uek(X)  /\ep} \,\Ke(X,Y)}{d\mu(X) }
\right)  \, ,  \\[6pt]
&   \uekp(X) = - \ep \, W_0 \left( \frac{g}{\ep} \, u^k(X) \, 
 \SI{  e^{   \vekp(Y))/\ep} } {  \Ke(X,Y) \,  d\sigma(Y) }
\right)    \, ,    \\[8pt]
& g\, u^{k+1}(X)  = \dfrac{  \ep \, \log(u^k(X)) -\uekp(X)}{ \SI{ u^k(X') \,  \,\Ke(X,X') }{d\mu(X')} }.
\end{align}
Numerically, we observe a better convergence 
compared with (\ref{AD1}-\ref{AD2}). We also note that 
Sinkhorn does not require a choice of relaxation parameter $t$. 
%\jfnote{Yes the a la Sinkhorn approach gives us faster convergence. In fact for the simple update rule 3.17/3.18, it take many more iterations to achieve tolerance even below 10-5, even with a large step parameter (learning rate, or t in 3.17/18).
%I do not seem to be able to get this simple stepping to go below 10-7 tolerance, even is I shrink the step. 
%Of course there are more complex stepping scheme that should be able to overcome this.
%}

\section{Numerical study}

\subsection{Discretisation and Algorithm}

The physical domain $\Omega$  is discretised using a time independent 
  (i.e. Eulerian) regular Cartesian grid $(X_i)_{i \in \llbracket 1, N \rrbracket }$. We make the approximation $ \mu = \frac{1}{N}\sum_i \delta_{X_i}$  and 
\begin{equation} 
\label{ED}
(h_t \mu )=    \frac{1}{N} \sum_i  h_{t,i} \, \delta_{X_i}, 
\end{equation} 
where $ (h_{t,i})_i$ is the height $X$-grid function.  Given a
continuous initial height measure $h_0$, the initial geostrophic
measure is $\sigma_0 = (Id+ \frac{g}{f^2} \, \nabla h_0)_\# \, h_0 \,
\mu $.  This is approximated as
\begin{equation} 
\label{s0} 
\sigma_0 =  \frac{1}{N} \sum_i  h_{0,i} \, \delta_{\Y_{0,i}}, 
\mbox{  where $\Y_{0,i}  = \{ Id+ \frac{g}{f^2} \,  \nabla h_0  \} (X_i)$.}
\end{equation}

 The family $(\Y_{0,i})_i$  specifies the initial position of the Lagrangian particles $(\Y_{t,i})_i$ following  the flow (\ref{FlP}). This flow is the  Lagrangian discretisation for $\sigma_t  = (\Y_t)_\# \sigma_0 $ according to
 \begin{equation}
 \label{LD}
 \sigma_t =  \frac{1}{N} \sum_i  h_{0,i} \, \delta_{\Y_{t,i}}.
 \end{equation}
The dynamics of the particles is given by the debiased velocity (\ref{FlP}),
\begin{equation}
\label{FlP2} 
\pp{}{t}\Y_{t,i}  = f\, J \cdot  \left\{  \nabla \v_{t,\ep}(\Y_{t,i})  - \nabla \v^S_{t,\ep}(\Y_{t,i}) \right\},
\end{equation} 
 for all $i$. Here, for all $t$, $(\nabla \v_{t,\ep}, \nabla \v^S_{t,\ep})$ is given statically  by (\ref{BDC}) and (\ref{BDC2}). The discretised potentials in the expression of the maps are given by the Sinkhorn solutions of  (\ref{SPP}) and (\ref{PSIS}) with marginals given and 
discretised as (\ref{ED}-\ref{LD}).  \\

Finally the system of ODEs (\ref{FlP2}) is discretised and integrated in time using a time stepping schemes, e.g. Euler, Heun (RK2) or RK4.
Figure (\ref{convergence_in_timestepping}) shows the number of Sinkhorn iterations to obtain converged potentials at successive time steps. 
When using the previous time step solutions as a warm initialisation, we see improved convergence.

\begin{figure}[h]
    \centering
    \includegraphics[width=\linewidth]{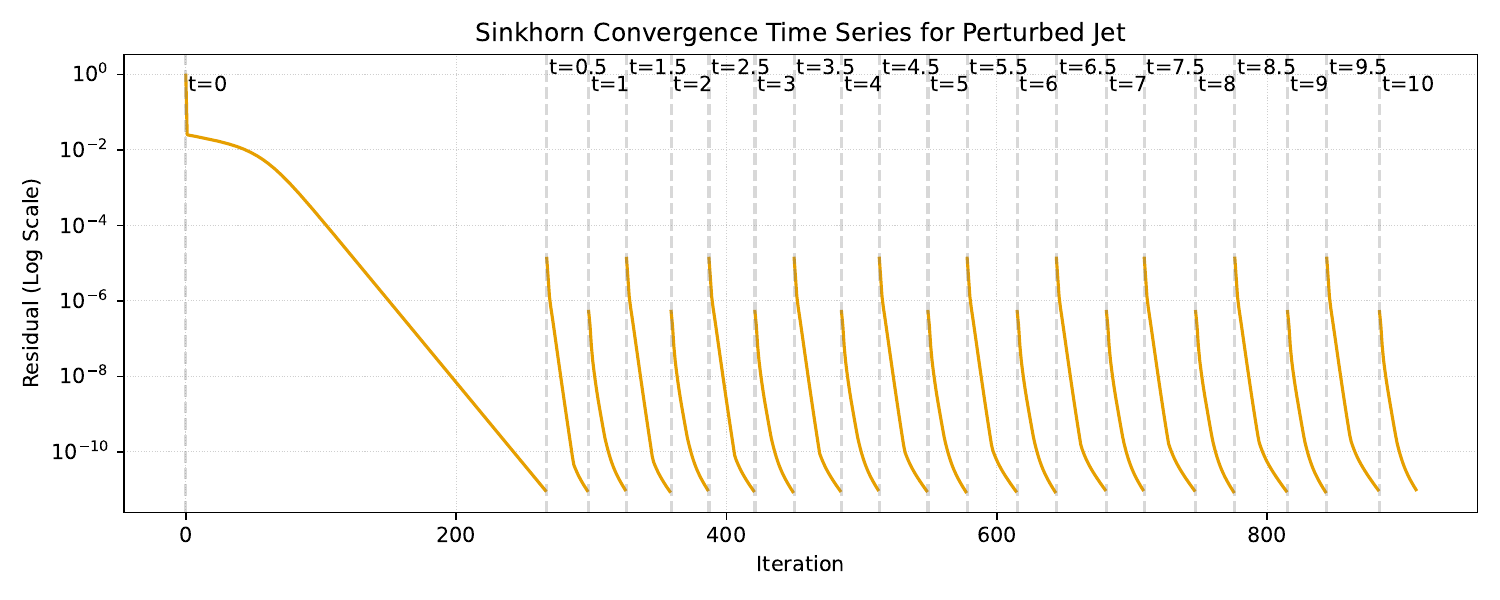}

    \caption{Illustration of the convergence of the Sinkhorn iterates for the perturbed jet case (see below) with \(dt=0.1, \varepsilon=0.01, \alpha = 0.001\) and the reinitialisation strategy we employ, which recycles the potentials from the previous iterations (when including dynamics). The index \(t\) represents the number of time steps in our Sinkhorn iterates. The residual is measured as \(\max\{|\phi^{k+1} - \phi^{k}|_{\infty}, |\psi^{k+1} - \psi^{k}|_{\infty}\}\) at iteration $k$, with the iteration terminating after the residual goes below a tolerance (1e-11 here). Timestep values halfway between integers correspond to the second intermediate stage in the two stage Runge-Kutta method used.}
    \label{convergence_in_timestepping}
\end{figure}

\subsection{Test case descriptions}
In this section we describe the test problems that we used in our numerical study, focussed on jet solutions in a periodic domain.
The 2D (horizontal) domain $\Omega$ \( = [0,1]\times[0,1]\) is periodic in the \(x_1\) direction, and we use a suitable periodic 
(squared) distance in the formulation of the optimal transport problem,
\begin{equation}
\label{HK2} 
c(x,y) = \min \left\{ |x_1 -y_1 |^2, |x_1 -y_1+1 |^2, |x_1 -y_1-1 |^2 \right\}  + |x_2-y_2 |^2.
\end{equation}
Before each optimal transport calculation, particles are remapped into the primary domain $0\leq x_2 \leq 1$.
Rigid boundary conditions on the bottom  and top  $x_2=0,1$ boundaries
are implicitly enforced by the transport target domain.  \\

In all cases we select units with Coriolis and gravity constants \(\hat{f}=1, \hat{g}=0.1\). 
These are appropriately chosen to occupy the SG rotationally dominated regime, and correspond to Burger numbers and Rossby numbers both equal to 0.1, which is required to be small for SWSG.

Our first testcase uses  
a  $x_1$-independent $h_0$ initialisation that satisfies CSP (assumption \ref{csp}). It generates a  stationary (in time) jet height  profile that remains in geostrophic balance. Indeed, the geostrophic speed (\ref{eq:geobal}) is 
the rotation of the $h$ gradient and therefore $u_{2,g}$ vanishes  and $u_{1,g} (x_1,x_2)  = C(x_2) $ 
is constant in $x_1$.  The density of $\sigma_0$ is therefore also independent of $x_1$ as well as the two transport maps $\nabla P_0$ and $\nabla Q_0$. Speeds and  densities are constant in $x_1$, so the $x_1$ independence is preserved in time by the geostrophic dynamics.  Geostrophic particles retain their initial $x_1$ independent speed, there is no  
acceleration, hence $U = U_g$( as per Equations (\ref{eq:3d sg u}-\ref{eq:3d sg v})). 
Our stationary jet is obtained from the initial height,
\begin{align}  %\tag{T1}
h_0(x_1, x_2) = a \tanh(b(x_2-c)) + d ,
\end{align}
with damping term, \(a\), slope scale term, \(b\), shift away from zero, \(c\), average height, \(d\). 
Here, Cullen Stability Principle holds provided that  \(\frac{3 \sqrt{3} f}{4 g} > ab^2\). 
We consider two cases, a shallow jet and a steeper jet. For both cases, \(a=0.1, c=0.5, d=1.0\), but \(b=10\) for the steeper jet and \(b=5\) for the shallow jet. 
This height field yields a pressure gradient and a subsequent stationary jet. 

Our second  test case corresponds to a perturbation of the initial height with a  Gaussian bump,
\begin{align} %\tag{T2}
    h_0(x_1, x_2) = a \tanh(b(x_2-c)) + d + \frac{\alpha}{2 \pi \sigma_0^2}e^{-\frac{1}{2\sigma_0^2} \, 
    | x_1 - \mu_1|^2 + |x_2 - \mu_2| ^2 } \label{eq:perturbed_jet_height_function}.
\end{align} 
Throughout we fix \(\mu_1=0.5, \mu_2=0.3, \sigma_0 = 0.1, \alpha = 0.001\).
Now necessarily  \(u_{2,g} \neq 0 \), and
instabilities will form leading to waves and eventually front formation.  Cullen stability principle holds provided \(\sigma_0, \alpha\) are chosen to keep \(P\) convex. For this, it is sufficient to choose \(\alpha\) small; here we take \(\sigma=0.1\), and \(\alpha=0.001\).
%\footnote{\jfnote{I'm struggling to find neat conditions on sigma and alpha, like for the stable jet. But for positive  trace and det we find;
%\(2 - 2\tanh sech^2(10(y-0.5)) + 0.1 \frac{\alpha}{2\pi\sigma^2}((\frac{y-0.3}{\sigma^2})^2 + (\frac{x-0.5}{\sigma^2})^2 - \frac{2}{\sigma^2})e^{-\frac{(x-0.5)^2+(y-0.3)^2}{2\sigma^2}} > 0\) and \((1 - 2\tanh sech^2(10(y-0.5)) + 0.1 \frac{\alpha}{2\pi\sigma^2}((\frac{y-0.3}{\sigma^2})^2 - \frac{1}{\sigma^2})e^{-\frac{(x-0.5)^2+(y-0.3)^2}{2\sigma^2}} )(1 + + 0.1 \frac{\alpha}{2\pi\sigma^2}((\frac{x-0.5}{\sigma^2})^2 - \frac{1}{\sigma^2})e^{-\frac{(x-0.5)^2+(y-0.3)^2}{2\sigma^2}}) -(0.1 \frac{\alpha}{2\pi\sigma^2}((\frac{y-0.3}{\sigma^2})(\frac{x-0.5}{\sigma^2}))e^{-\frac{(x-0.5)^2+(y-0.3)^2}{2\sigma^2}})^2 > 0\)}}
The profile begins as a stationary jet while the perturbation grows, leading to the formation of waves.
Eventually, these waves break, forming fronts.

\begin{figure}[h]
    \centering
    \includegraphics[width=0.9\linewidth]{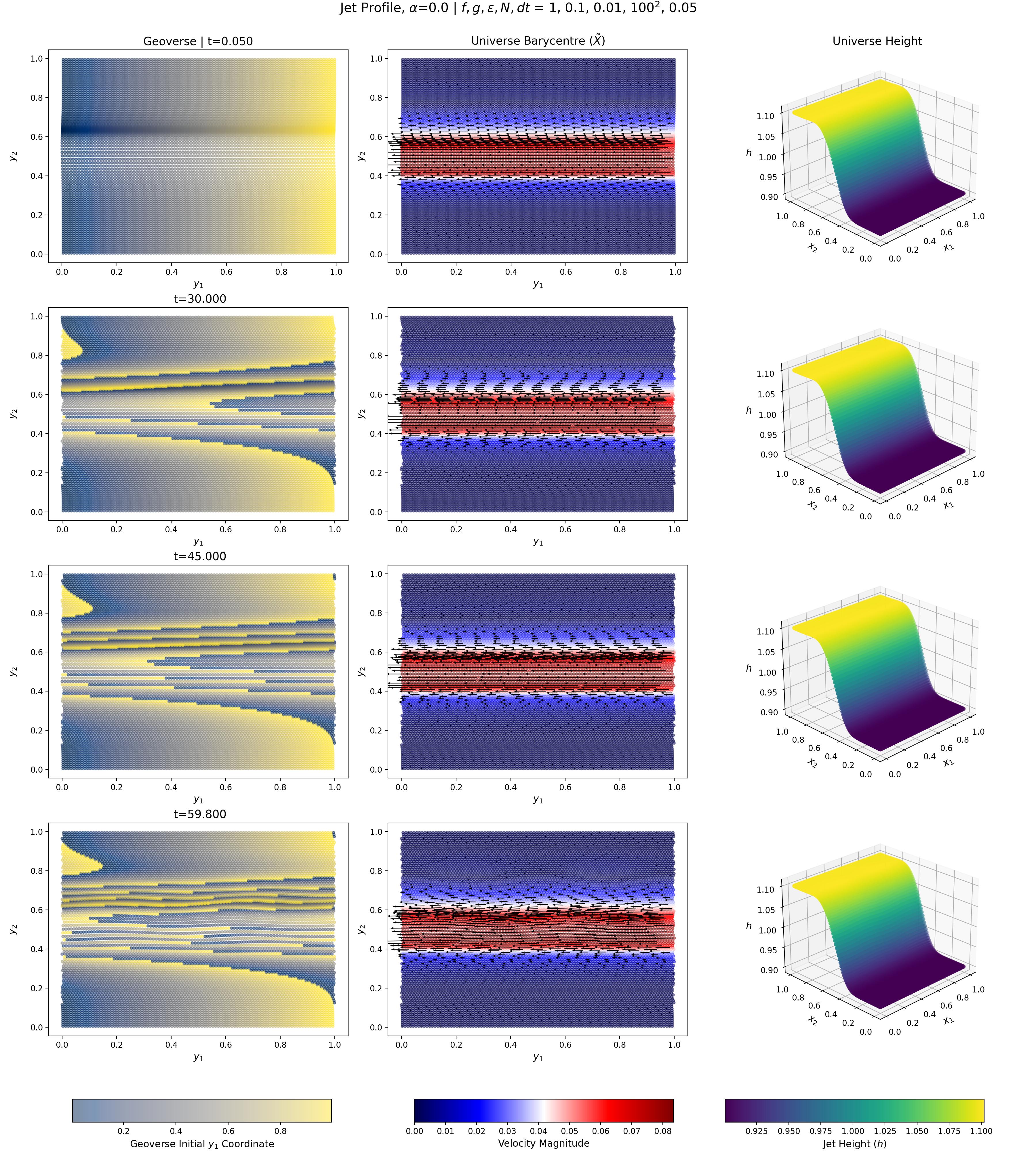}
    \caption{Stationary Jet profile, with parameters (a,b,c,d) = (0.1, 10, 0.5, 1.0). The Figure illustrates the integration (using Heun's stepping) at \(t = 0, 30, 60\). The left column shows the (diagnostic) points in physical space, coloured by their initial \(y_1\) position illustrating mixing over time. The middle column displays the corresponding barycentric projection and points are coloured by the approximate universe velocity from a 2nd order finite difference method, with illustrative quivers. %Note that the geostrophic and true velocities should be close, as the ageostrophic component should remain small. 
    The right column shows the reconstructed height fields.
    %which for the non-perturbed jet should remain stationary.\jfnote{ The first and last time are off 0 and 60 because of the central difference 'true' velocities.} }
    }
    \label{jet_figure}
\end{figure}

\begin{figure}[h]
    \centering
    \includegraphics[width=0.9\linewidth]{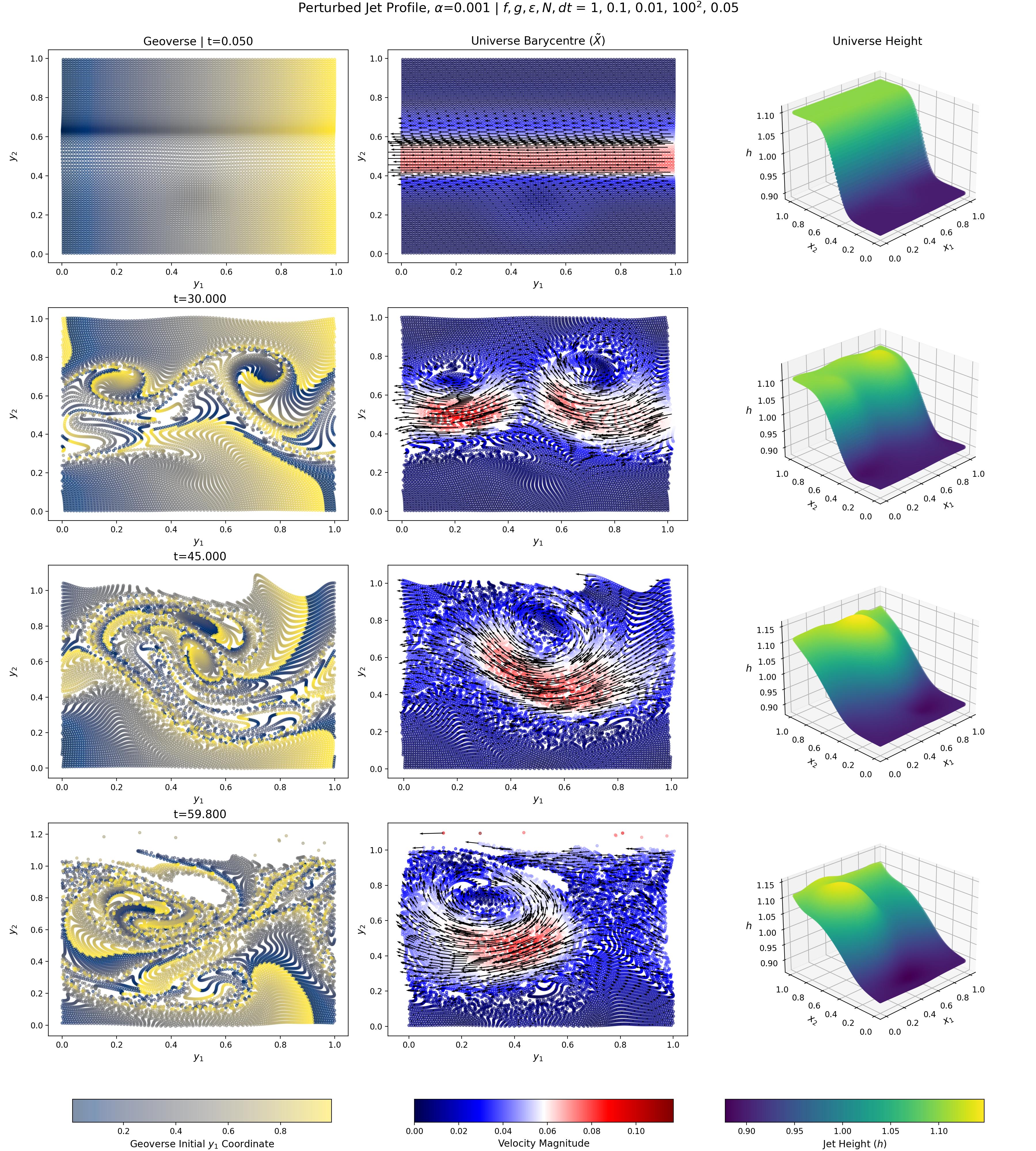}
    \caption{Perturbed jet profile, with parameters (a,b,c,d,) = (0.1, 10, 0.5, 1.0) and (\(\mu_{x_1}, \mu_{x_2}, \sigma, \alpha\)) = (0.5,0.3,0.1,0.001). The Figure illustrates the integration (using Heun's stepping) at \(t = 0, 30, 45, 60\). The left column shows the (diagnostic) points in physical space, coloured by their initial \(y_1\) position illustrating mixing over time. The middle column displays the corresponding barycentric projection and points are coloured by the approximate universe velocity from a 2nd order finite difference method, with illustrative quivers. 
    %Note that the geostrophic and true velocities should be close, as the ageostrophic component should remain small. 
    The right column shows the reconstructed height fields for the perturbed jet, where fronts start to form. 
    %\jfnote{ The first and last time are off 0 and 60 because of the central difference 'true' velocities.}}
    }
    \label{perturbed_jet_figure}
\end{figure}

\begin{figure}[h]
    \centering
    \includegraphics[width=1.0\linewidth]{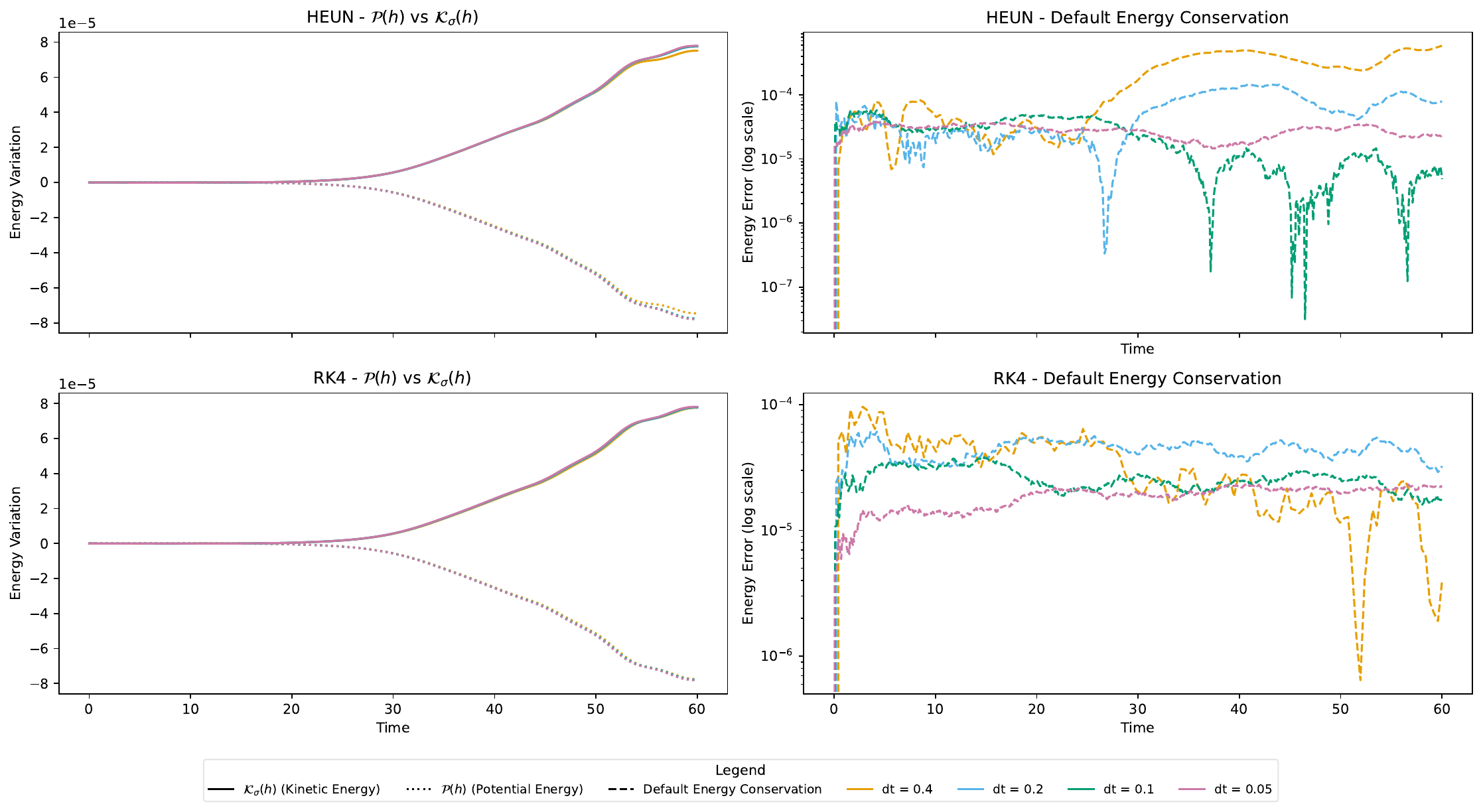}
    \caption{Energy conservation test for the perturbed jet, using the debiased Sinkhorn divergence. The rows compare two time-stepping schemes: Heun (second-order Runge-Kutta) and classical Runge-Kutta 4. The left column shows how potential and kinetic energy vary from their initial values over time. All four time steps closely overlap. The right column plots the absolute, normalised default energy error (See equation \ref{eq:default_energy}) on a semi-logarithmic scale. The entropy contribution, not shown here, stays below the set tolerance of $1e-11$. The energy variation is defined as the current energy (kinetic or potential) minus the initial energy.  
    }
    \label{energy_conservation_for_perturbed_jet}
\end{figure}

%The normalised default energy conversation is shown, where the non-exchangeable background energy corresponding to the available potential energy from the uniform Lebesgue height profile (\(\frac{g_0}{2}\norm{h_{uni}}\)) is removed.  Hence the left column shows, using the energy notation \ref{primalreg} with \(S_{\varepsilon}\) not \(OT_{\varepsilon}\),
%\begin{align}
%    \frac{|\hat{\mathcal{E}}_{\sigma,\varepsilon}(h_t)  - \hat{\mathcal{E}}_{\sigma,\varepsilon}(h_0)|}{\hat{\mathcal{E}}_{\sigma,\varepsilon}(h_0) - %\frac{g_0}{2}\norm{h_{uni}}}. \label{eq:default_energy}
%\end{align}
%\changes{What is the hat on $\Ec$ for ????}
%\jfnote{The hat was for non-dimensional energy. Do we have notation for total energy with sinkhorn divergence not OT, I don't think so.}

\begin{figure}[h]
    \centering
    \includegraphics[width=1.0\linewidth]{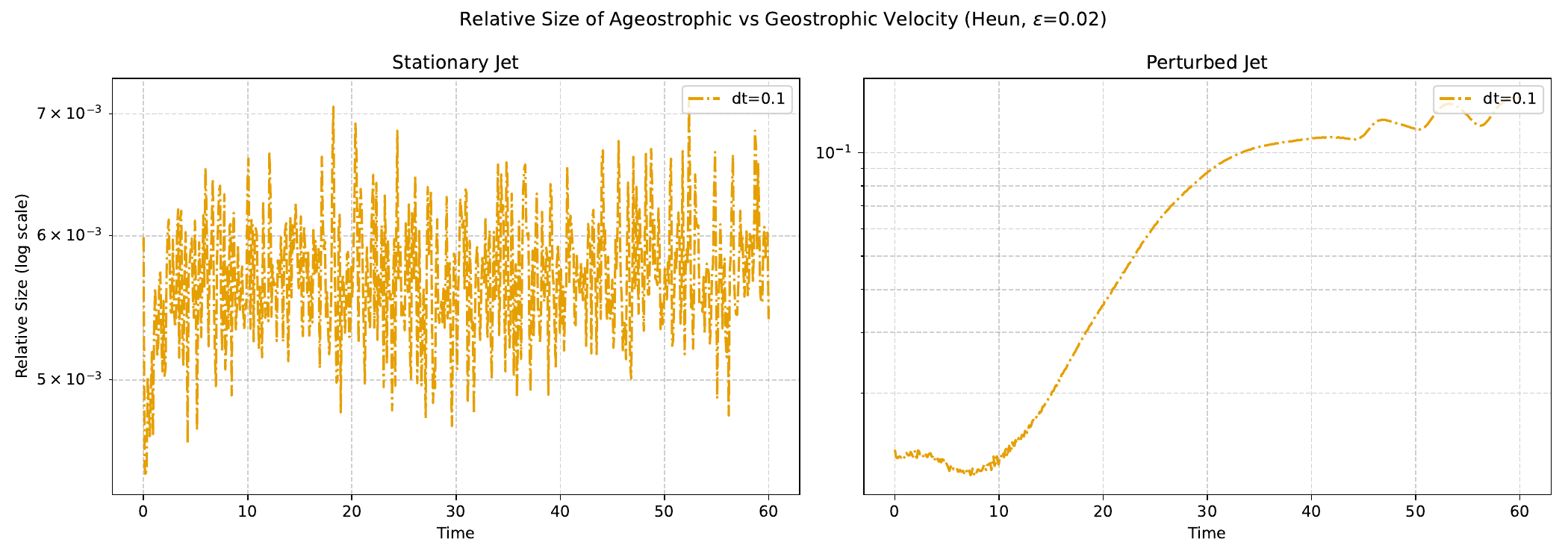}
    \caption{Time evolution of the relative size of ageostrophic/geostrophic velocities  the jet and perturbed jet 
    (see section \ref{age}).
    }
    %To remain in the SG  limit the ageostrophic velocity is expected to remain small relative to the geostrophic velocity. That is, given \(U_{g,t} = J (G_t- \tilde X(G_t))\) and \(U = \frac{X_{t+1} - X_{t-1}}{2 dt}\), we measure the relative size of the ageostrophic component \(\norm{U - U_g}_{2} / \norm{U_g}_2\) against the geostrophic over time using Heun stepping. Recalling \(U - U_g = U_{ag}\). \(\varepsilon = 0.02, \alpha =0.001\).}
    \label{ageostrophic_error_for_jet}
\end{figure}

\subsection{Numerical results}
\label{ne}
First we investigate the accuracy of the reconstructed height field corresponding to the steady initial jet profiles.
At time $0$, the data $\sigma_0$ is constructed from the analytically provided initial height $h_0$, see (\ref{s0}).
Provided it satisfies CSP, the solution of  (\ref{e3}) $(\u,\v)$ reconstructs  $h_0(X)$ as $-f^2\,\u(X)/g $  and the geostrophic velocity in geostrophic coordinates 
$U_{g,t}(\nabla Q_0(X_t))  =  -g/f \, J \, \nabla h_0(Y :=X+\nabla \, h_0(X))$ as $ J \, \nabla \psi(Y)$. \\

Figure \ref{fig:height_stationary_error} 
% \changes{ I suggest to change jet into stationary jet and Trend (slope) in just slope $1.5$ in the figure}
shows the height error, for different initialisations and as a function of $\ep$ ($N$ given as above), when using $OT_\ep$ as an approximation of $W_2^2$, 
\[
E^h_{OT_\ep} = \sqrt{S_{0.01}(\frac{1}{N}\sum_{i=1}^N( - \u_{\ep,i} \delta_{X_{i}}, \frac{1}{M} \sum_{k=1}^M {h_{0,k} \delta_{X_{k}})} },
\]
 and 
\[
E^h_{\Sc_\ep} = \sqrt{\Sc_{0.01}(\frac{1}{N}\sum_{i=1}^N( (\u^S_{\ep,i} -\u_{\ep,i} )\delta_{X_{i}}, \frac{1}{M} \sum_{k=1}^M {h_{0,k} \delta_{X_{k}})} }
\]
when using the debiased $\Sc_{\ep}$.
 We measure the errors using a fixed $\ep = 0.01$ Sinkhorn divergence $\Sc_{0.01}$  as a loss between these discrete measures.
 In both cases, the exact $h_0$ is approximated on a fine grid ($M=640^2 >>N$). Debiasing yields lower errors but does not seems to improve the order of convergence, which is roughly $O(\ep^{1.5})$.  The steeper (harder) jet profile produces larger errors as expected. \\

In Figure \ref{fig:jet_phases}, we investigate the approximation error in $\ep$ at time 0 for the velocity of the geostrophic flow $\pp{}{t}\Y_t  = U_{g,t}(\nabla Q_0^{-1}(\Y_t))$. Since we use a Lagrangian discretisation, we use the transport loss $\Sc_{0.01}$ as a proxy for 
$W_2^2$ to measure errors, but this time in the (4D) phase space 
$(\Y_0, U_{g,0}(\Y_0) ) $. The plotted errors, for different initialisations are
\[
E^U_{OT_\ep} = \sqrt{\Sc_{0.01}\left(\frac{1}{N}\sum_{i=1}^N( h_{0,i}\, \delta_{ (\Y_{0,i}, - \frac{g}{f}\, J \cdot \nabla \v_{0,\ep}(\Y_{0,i}))} , \frac{1}{M} \sum_{k=1}^M h_{0,k} \, \delta_{ ( \nabla Q_0(\X_{0,k}), U_{g,0}(\X_{0,k}) )} \right)} ,
\]
when using $OT_\ep$ as an approximation of $W_2^2$, and 
\begin{align*}
E^U_{\Sc_\ep} &= \sqrt{\Sc_{0.01}( \mu_N, \nu_M )}, \\
\mu_N &= \frac{1}{N}\sum_{i=1}^N( h_{0,i}\, \delta_{ (\Y_{0,i}, - \frac{g}{f}\, J \cdot \nabla \left( \v_{0,\ep}(\Y_{0,i})) -\v^S_{0,\ep}(\Y_{0,i})) \right)} , \\
\nu_N &= \frac{1}{M} \sum_{k=1}^M h_{0,k} \, 
\delta_{ ( \nabla Q_0(\X_{0,k}), U_{g,0}(\X_{0,k}) )}, \\
\end{align*}
when using the debiased $\Sc_{\ep}$. The fine  $Y_{0,k}$ discretisation 
is constructed via Hoskins' transformation of a fine $X_{k}$ grid.
The debiased  version achieves a lower error but convergence rate is now $O(\ep^{0.75})$ except for the shallow jet case with a {weaker} gradient.

\begin{figure}
    \centering
    \includegraphics[width=0.5\linewidth]{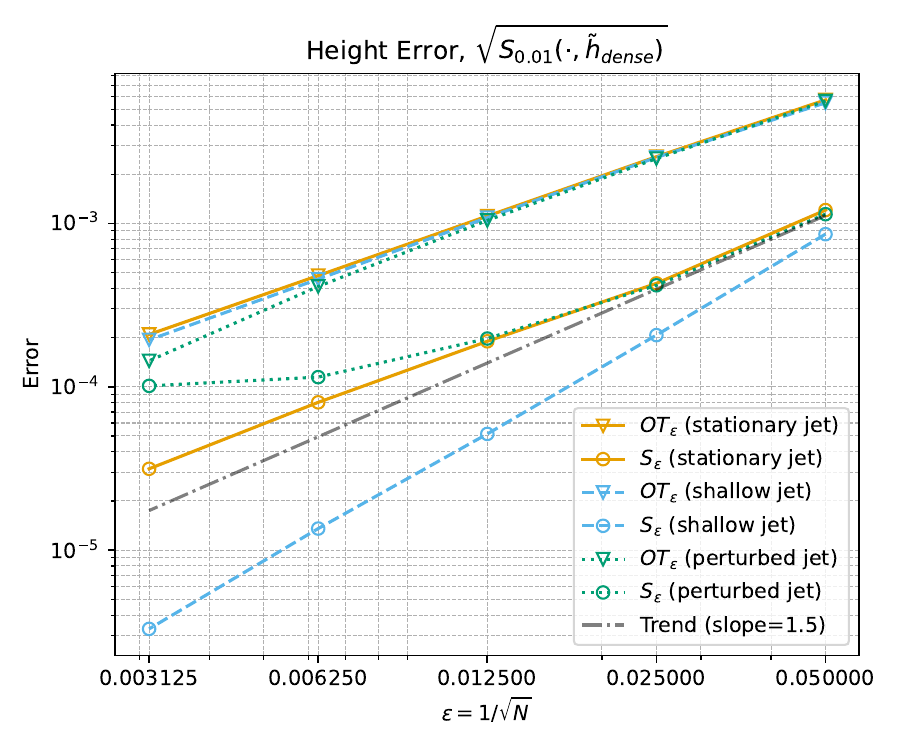}
    \caption{Convergence plots $\ep \rightarrow 
    E^h_{OT_\ep}$ and  $\ep \rightarrow E^h_{\Sc_\ep}$ (see section \ref{ne}) at time step $0$ for different initialisations. Note that the shallow jet is also stationary.} 
    \label{fig:height_stationary_error}
\end{figure}

\begin{figure}[h]
    \centering
    \begin{minipage}[b]{0.333333\textwidth}
        \centering
        \includegraphics[width=\textwidth]{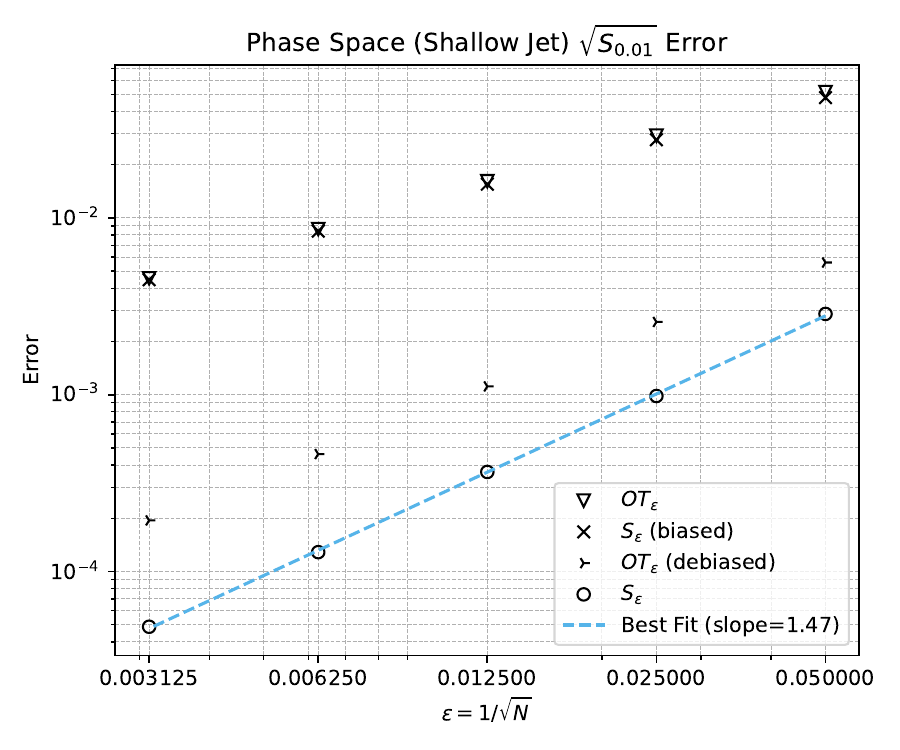}
    \end{minipage}\hfill
    \begin{minipage}[b]{0.333333\textwidth}
        \centering
        \includegraphics[width=\textwidth]{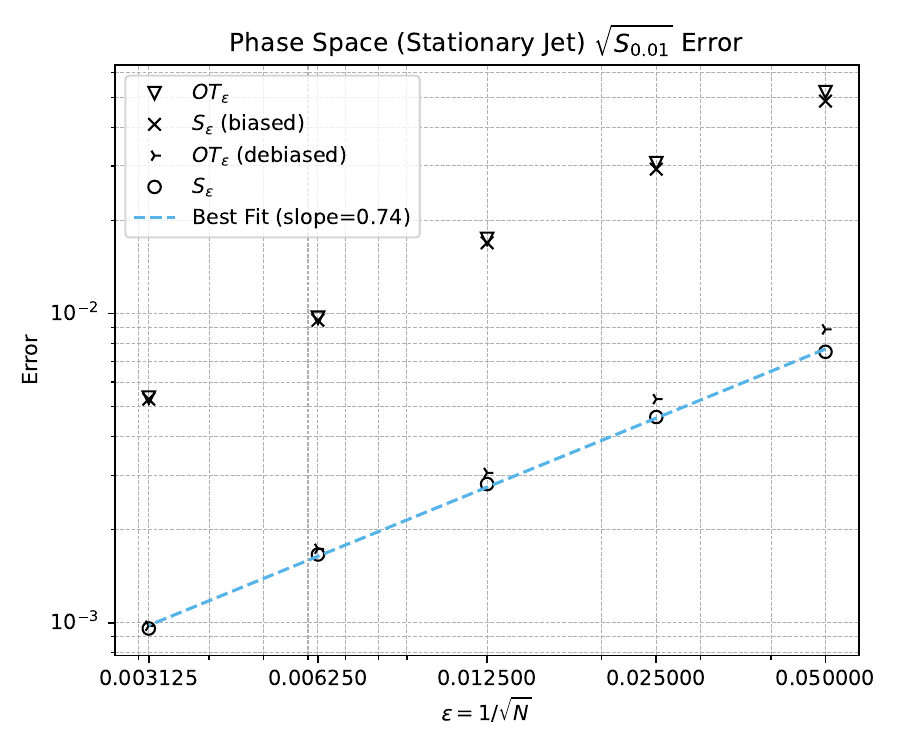}
    \end{minipage}\hfill
    \begin{minipage}[b]{0.333333\textwidth}
        \centering
        \includegraphics[width=\textwidth]{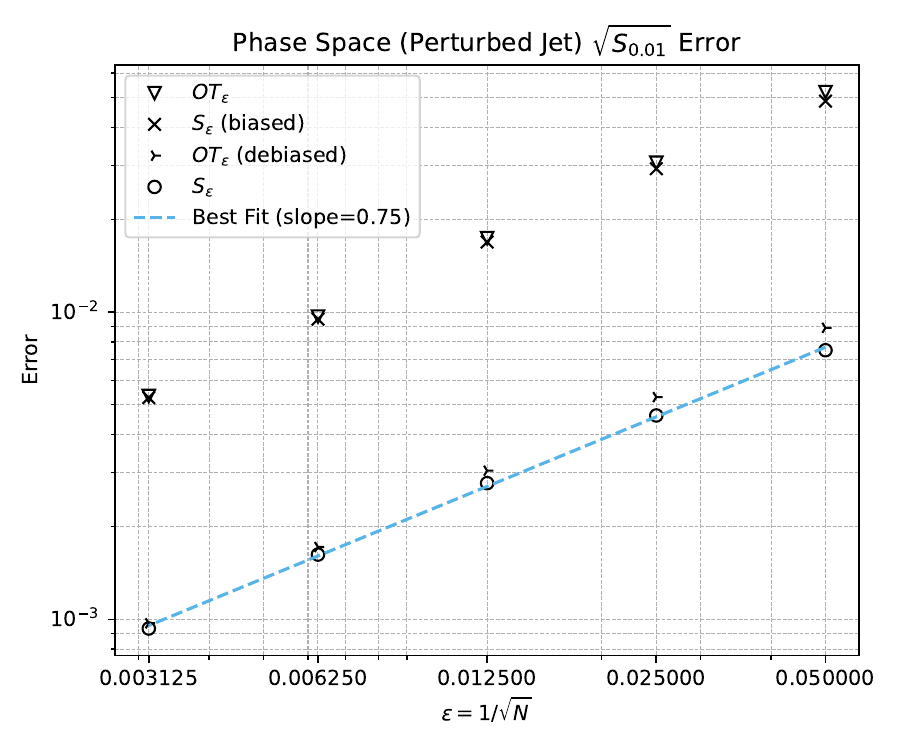}
    \end{minipage}
    \caption{
   Convergence plots $\ep \rightarrow 
    E^U_{OT_\ep}$ and  $\ep \rightarrow E^U_{\Sc_\ep}$ (section \ref{ne}) at time step $0$ for different initialisations. Note that the shallow jet is also stationary.}
    \label{fig:jet_phases}
\end{figure}

Figure \ref{jet_figure} demonstrates the stationary jet with geostrophic velocities (shown by the colour and quivers).
In the central column the barycentric projection into the physical domain shows how physical trajectories are formed, and then the 3D plot illustrates the reconstructed height profile of the tanh curve.
Clearly the stability of the jet is maintained with only small oscillations around the base state observed, due to asymmetries from the Lagrangian discretisation.

Figure \ref{perturbed_jet_figure}  demonstrates the nonstationary behaviour of the perturbed jet, leading to front-like structures.
This is evident in the final two frames, where the large velocities demonstrate the front being supported by strong geostrophic winds. 
This behaviour is also reflected in the height profile, where a bulge moves along the underlying perturbed jet.

To begin studying the correctness of these numerical solutions, we consider if the total energy is conserved, and explore its breakdown into potential and kinetic energy. The stable jet approximately conserves energy with variations on the scale of 1e-7 (not shown).
Figure \ref{energy_conservation_for_perturbed_jet} illustrates shows the energy exchange and conservation for the perturbed jet, when using the debiased Sinkhorn divergence methodology, in terms of
the normalised energy
\begin{align}
    \frac{\mathcal{E}_{\sigma}(h_t) - \mathcal{E}_{\sigma}(h_0) }{\mathcal{E}_{\sigma}(h_0) - \mathcal{E}_{\mathcal{U}} , }\label{eq:default_energy}
\end{align}
where \(\mathcal{E}_{\mathcal{U}}\) is the minimum background energy corresponding to a uniform height profile with no velocity. The spatial semidiscretisation conserves energy exactly, 
so any energy errors will arise from time discretisation, or by truncation of the iterative scheme to find the optimal transport solution.
Hence, we vary the size of time steps and the stepping scheme (2nd order Heun or 4th order Rk4).
Crucially, the system appears to converse its energy with variations in  the default energy conversation on the order of 1e-5.
There is marginal improvement in conservation from Heun to RK4, but a greater contrast in stability of energy is through smaller time stepping.
Overall all approaches keep the relative energy error small, around 1e-4. 

We remark that the method generates two approximations of the height. 
The (Eulerian, prognostic) grid function $(h_{t,i})_i$ on the right  in Figure \ref{perturbed_jet_figure}.  
A (diagnostic) weighted point cloud Lagrangian sampling, corresponding to the push-forward of geostrophic density sampling (\ref{LD}) by the barycentric map approximating (\ref{A1}),
given by
\begin{equation}
\label{rch}
   \frac{1}{N} \sum_i  h_{0,i} \, \delta_{ \X_{t,i} = \Y_{t,i} + \nabla \left( \v_{t,\ep}(\Y_{t,i})  - \v^S_{t,\ep}(\Y_{t,i}) \right) },
\end{equation}
is shown in the middle panel of 
 Figure \ref{perturbed_jet_figure},
The weights are fixed at initialisation and cannot capture 
large density variations. This is seen in particular in the last two 
 snapshots where  a ``hole'' appears in the domains in the last two 
 snapshots. This occurs because of the contractive nature of the barycentric mapping which is a combination of finite size effects and the entropic regularisation. 
 We emphasise that these reconstructions of Lagrangian trajectories are purely diagnostic, and are not involved in the SG solution algorithm; the true map is best approximated 
 by the many-to-many map defined by the discrete optimal coupling $\pi$.
 %as well as some barycentric points in the universe leaving the [0,1] box.  These errors arise due to the discretisation process and the boundary of the geoverse folding in on itself.
%\st{This, in turn, ruins the convexity of the domain we are mapping from, which results in a hole forming in its image.}
%\jfedit{The weights on the geoverse point cannot change and there is a front forming, hence points need to disperse below the font (creating the bulge in the height field). }
%\changes{ I am not sure of this in the eady slice we had stuff like that, here $h$ would still fill the hole but maybe with small mass, then the contraction of the Entropic maps is not well debiased m this was also observed in the eady slice ...}. \\

\label{age} The SG approximation of the SW equations neglects the acceleration of the 
ageostrophic component of the velocity 
$U_{t,ag} := U_t - U_{t,g}$ (see Section \ref{sgsw}). In Figure  \ref{ageostrophic_error_for_jet} we 
measure numerically  the time evolution of \(\norm{U_{t,ag}}_2 / \norm{U_{t,g}}_2 \), 
the relative size of the ageostrophic versus geostrophic velocity.
The full velocity needed to compute $U_{t,ag}$ is approximated via
second order central finite difference in time
\(U_t(\X_t) = (\X_{t+1} - \X_{t-1})/({2 dt})\) (see remark \ref{velo}).
For the stationary jet (on the left) the velocity is purely geostrophic 
(see the discussion on the initialisation above) and for the perturbed jet (on the right) the ratio 
remains around 0.1.

At a later positive time $t$ we do not have a 
a reference analytical solution for the 
non stationary  perturbed jet. Instead,
we examine the ``pseudoconvergence" towards a fine grid solution. Proposition \ref{Emsr} provides convergence results in the continuous case but little information on the convergence rates. 
We follow \citet{berman} (as discussed in \citet{BCM24} in section 4.2), parameterising the space discretisation $N$ with $\epsilon$ according to  $N= 1/\varepsilon^2 $. 
 We stop the Sinkhorn iterations  when $\L^\infty$ difference of $1e-11$ in the potential increments is reached.

The fine grid solution corresponds to $N_0 = 2^{16}$. Figure \ref{fig:perturbed_sigma_convergence} shows convergence 
for the discrete  geostrophic density (\ref{LD}),
\[
E^\sigma_{OT_\ep,\Sc_\ep}(\sigma_t^N)  = \sqrt{\Sc_{0.01}(\sigma_t^N, \sigma_t^{N_0}) }.
\]
The added  $N$ upper script on $\sigma_t^N$ clarifies the 
the dependence on the discretisation, again
using the Sinkhorn divergence loss as a proxy for $W_2$.
Both the entropic $(OT_\ep)$ and debiased entropic $(\Sc_\ep)$ are tested.  

Figure \ref{fig:s_eps_pseudo_convergence} shows the pseudoconvergence in the reconstructed height on the $X$ grid, comparing the use of $ h^N = -f/g  \, \ue^N$  for the entropic $OT_\ep$ with $ h^N = -f/g  \, (\ue^N - \ue^{S,N}) $ for the $\Sc_\ep$ debiased approach. The fine reference solution $h^{N_0}$ is computed with the debiased approach. Since the $X$ grid is structured, we can also compute an $L^p$ loss combined with interpolation in addition to the $\Sc_{0.01}$ transport loss. 
We observe a pseudoconvergence of order $O(\ep^{3/2})$,
suggesting that the discretisation error dominates in the 
$t=0$ convergence against the exact solution (Figure \ref{fig:jet_phases}). We also remark that debiasing improves the accuracy but not the order of convergence.
This was already observed for the SG Eady slice test case 
\cite{BCM24}. 
Heuristically, it can be explained  by noticing that the 
transport correction of the symmetric part in \ref{SD} is asymptotically irrotational as $\ep \to 0$ (see remark 6 in \cite{BCM24}).

 \begin{figure}
     \centering
     \includegraphics[width=0.5\linewidth]{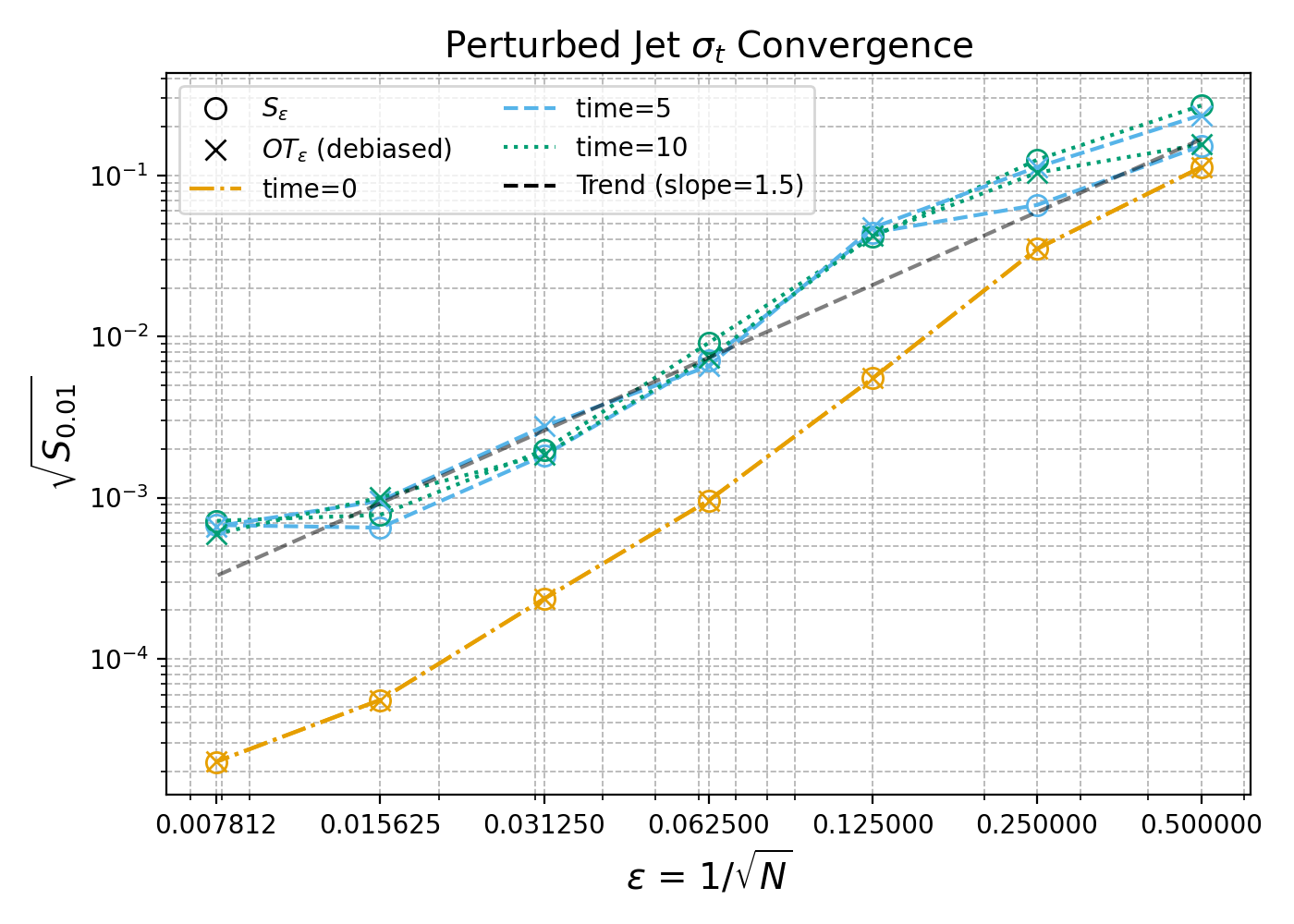}
     \caption{Pseudo convergence against a fine solution for the perturbed jet, this is 
     $E^\sigma_{OT_\ep,\Sc_\ep}(\sigma_t^N)$ for $N \nearrow 2^{16}$ and different times (section \ref{ne}). }
     \label{fig:perturbed_sigma_convergence}
 \end{figure}

\begin{figure}
    \centering
    \begin{subfigure}{0.49\linewidth}
        \centering
        \includegraphics[width=\linewidth]{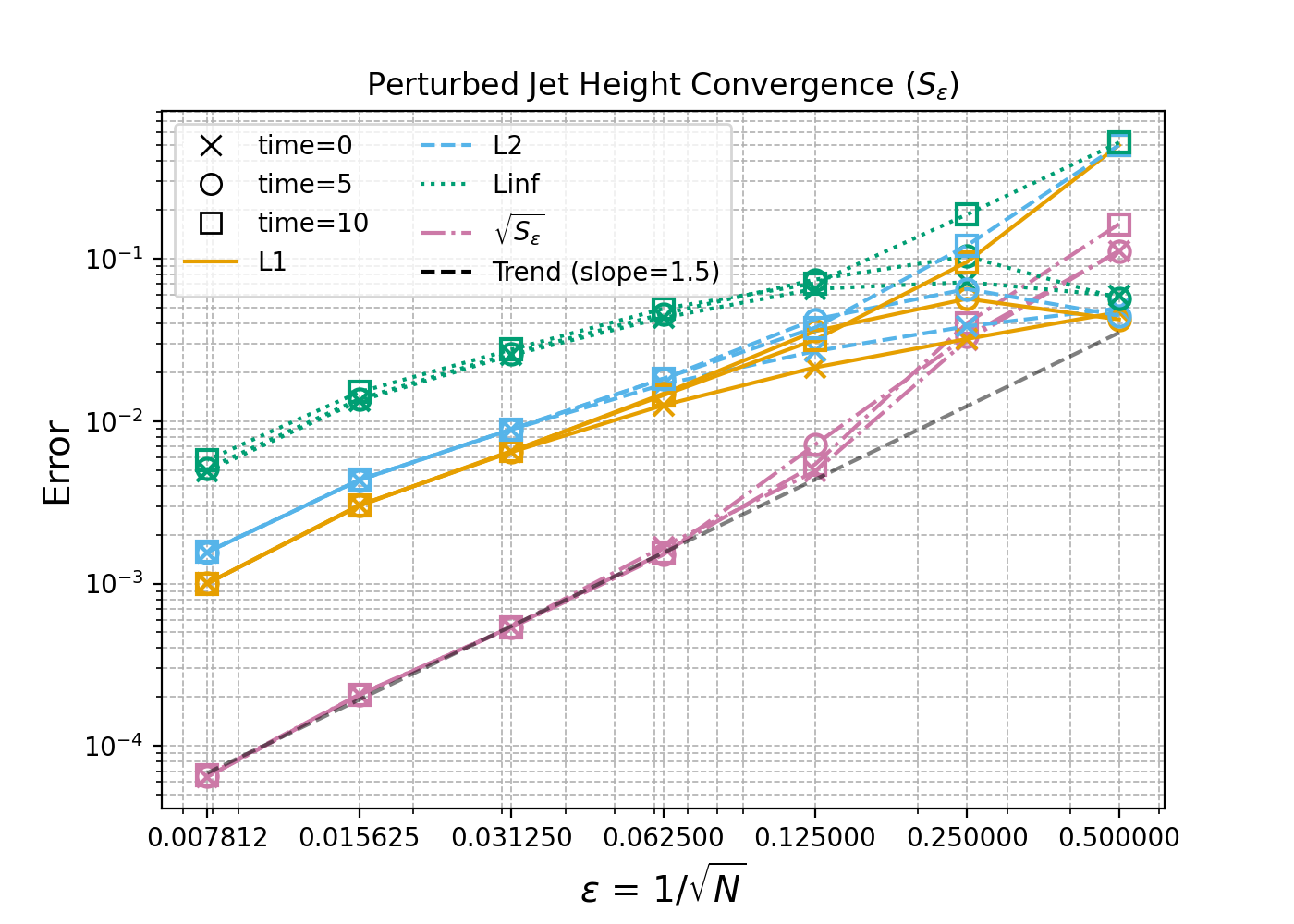}
    \end{subfigure}
    \begin{subfigure}{0.49\linewidth}
        \centering
        \includegraphics[width=\linewidth]{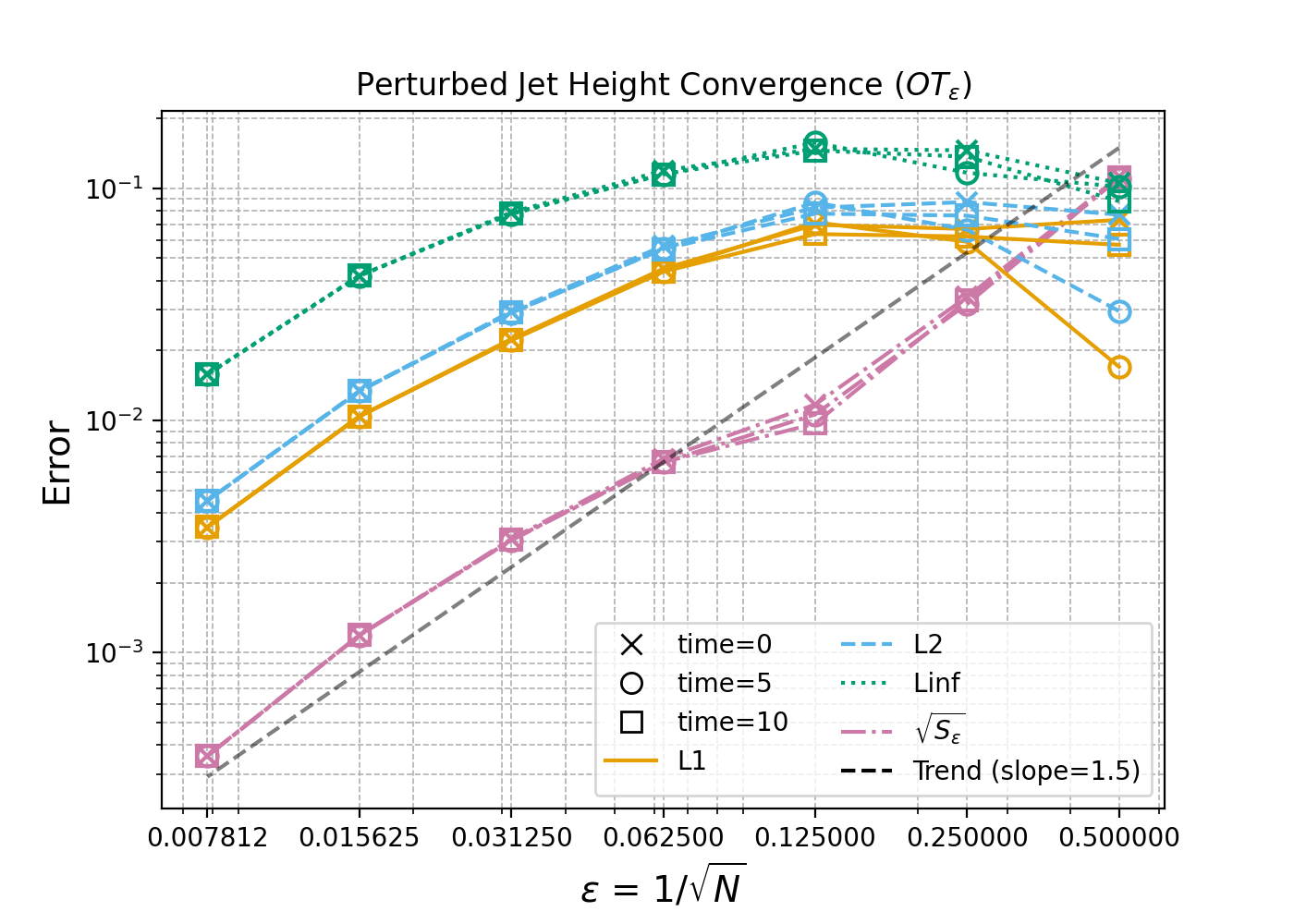}
    \end{subfigure}
    \caption{Same as Figure \ref{fig:perturbed_sigma_convergence} but for the reconstructed height and different losses.
(section \ref{ne}). } 
    \label{fig:s_eps_pseudo_convergence}
\end{figure}

\section*{Acknowledgements} 
JD Benamou gratefully  acknowledges the support of the
CNRS Imperial College Abraham de Moivre International Laboratory Fellowship.
J.J.M. Francis gratefully acknowledges the support of this work by the Natural Environment Research Council [grant number NE/S007415/1] and their CASE Studentship "Transport methods for the verification of Numerical Weather Prediction (NWP) forecasts" at the Met Office.

\bibliography{references,references_JD}

\begin{thebibliography}{39}
\expandafter\ifx\csname natexlab\endcsname\relax\def\natexlab#1{#1}\fi
\expandafter\ifx\csname url\endcsname\relax
  \def\url#1{\texttt{#1}}\fi
\expandafter\ifx\csname urlprefix\endcsname\relax\def\urlprefix{URL }\fi

\bibitem[{Ambrosio and Gangbo(2008)}]{ambrosio2008hamiltonian}
Ambrosio, L., Gangbo, W., 2008. Hamiltonian {ODE}s in the {W}asserstein space
  of probability measures. Communications on Pure and Applied Mathematics
  61~(1), 18--53.

\bibitem[{Benamou et~al.(2024)Benamou, Cotter, and Malamut}]{BCM24}
Benamou, J.-D., Cotter, C., Malamut, H., 2024. Entropic optimal transport
  solutions of the semigeostrophic equations. Journal of Computational Physics
  500, 112745.
\newline\urlprefix\url{https://www.sciencedirect.com/science/article/pii/S0021999123008410}

\bibitem[{Berman(2020{\natexlab{a}})}]{berman2020stability}
Berman, R.~J., Dec. 2020{\natexlab{a}}. Convergence rates for discretized
  {M}onge{\textendash}{A}mp{\`{e}}re equations and quantitative stability of
  optimal transport. Foundations of Computational Mathematics 21~(4),
  1099--1140.
\newline\urlprefix\url{https://doi.org/10.1007/s10208-020-09480-x}

\bibitem[{Berman(2020{\natexlab{b}})}]{berman}
Berman, R.~J., aug 2020{\natexlab{b}}. The {S}inkhorn algorithm, parabolic
  optimal transport and geometric {M}onge--{A}mp\`{e}re equations. Numerische
  Mathematics 145~(4), 771--836.
\newline\urlprefix\url{https://doi.org/10.1007/s00211-020-01127-x}

\bibitem[{Bourne et~al.(2022)Bourne, Egan, Pelloni, and Wilkinson}]{BourneP}
Bourne, D., Egan, C., Pelloni, B., Wilkinson, M., 2022. Semi-discrete optimal
  transport methods for the semi-geostrophic equations. Calculus of Variations
  and Partial Differential Equations 61, cvgmt preprint.
\newline\urlprefix\url{http://cvgmt.sns.it/paper/4811/}

\bibitem[{Bourne et~al.(2025)Bourne, Egan, Lavier, and
  Pelloni}]{bourne2025semi}
Bourne, D.~P., Egan, C.~P., Lavier, T., Pelloni, B., 2025. Semi-discrete
  optimal transport techniques for the compressible semi-geostrophic equations.
  arXiv preprint arXiv:2504.20807.

\bibitem[{Carlier et~al.(2017)Carlier, Duval, Peyr{\'e}, and
  Schmitzer}]{carlier2017convergence}
Carlier, G., Duval, V., Peyr{\'e}, G., Schmitzer, B., 2017. Convergence of
  entropic schemes for optimal transport and gradient flows. SIAM Journal on
  Mathematical Analysis 49~(2), 1385--1418.

\bibitem[{Carlier and Malamut(2024)}]{carliermalamut24}
Carlier, G., Malamut, H., 2024. Well-posedness and convergence of entropic
  approximation of semi-geostrophic equations.
\newline\urlprefix\url{https://arxiv.org/abs/2404.17387}

\bibitem[{Chizat(2017)}]{chizat2017unbalanced}
Chizat, L., 2017. Unbalanced optimal transport: Models, numerical methods,
  applications. Ph.D. thesis, Universit{\'e} Paris sciences et lettres.

\bibitem[{Chizat et~al.(2018)Chizat, Peyr{\'e}, Schmitzer, and
  Vialard}]{chizat2018scaling}
Chizat, L., Peyr{\'e}, G., Schmitzer, B., Vialard, F.-X., 2018. Scaling
  algorithms for unbalanced optimal transport problems. Mathematics of
  computation 87~(314), 2563--2609.

\bibitem[{Chizat et~al.(2020)Chizat, Roussillon, L{\'e}ger, Vialard, and
  Peyr{\'e}}]{chizatfaster}
Chizat, L., Roussillon, P., L{\'e}ger, F., Vialard, F.-X., Peyr{\'e}, G., 2020.
  Faster {W}asserstein distance estimation with the {S}inkhorn divergence.
  Advances in Neural Information Processing Systems 33, 2257--2269.

\bibitem[{Conforti and Tamanini(2021)}]{conforti2019formula}
Conforti, G., Tamanini, L., 2021. A formula for the time derivative of the
  entropic cost and applications.
\newline\urlprefix\url{https://www.sciencedirect.com/science/article/pii/S002212362100046X}

\bibitem[{Cullen(2018)}]{cullen2018use}
Cullen, M., 2018. The use of semigeostrophic theory to diagnose the behaviour
  of an atmospheric {GCM}. Fluids 3~(4), 72.

\bibitem[{Cullen and Gangbo(2001)}]{cullen2001variational}
Cullen, M., Gangbo, W., 2001. A variational approach for the 2-dimensional
  semi-geostrophic shallow water equations. Archive for rational mechanics and
  analysis 156, 241--273.

\bibitem[{Cullen and Maroofi(2003)}]{cullen2003fully}
Cullen, M., Maroofi, H., 2003. The fully compressible semi-geostrophic system
  from meteorology. Archive for rational mechanics and analysis 167, 309--336.

\bibitem[{Cullen and Purser(1989)}]{cullen1989properties}
Cullen, M., Purser, R., 1989. Properties of the {L}agrangian semigeostrophic
  equations. Journal of Atmospheric Sciences 46~(17), 2684--2697.

\bibitem[{Cullen and Roulstone(1993)}]{cullen1993geometric}
Cullen, M., Roulstone, I., 1993. A geometric model of the nonlinear
  equilibration of two-dimensional {E}ady waves. Journal of Atmospheric
  Sciences 50~(2), 328--332.

\bibitem[{Cullen(2006)}]{cullen2006mathematical}
Cullen, M. J.~P., 2006. A mathematical theory of large-scale atmosphere/ocean
  flow. World Scientific.

\bibitem[{Cullen(2007)}]{cullen2007modelling}
Cullen, M. J.~P., 2007. Modelling atmospheric flows. Acta Numerica 16, 67--154.

\bibitem[{Cuturi(2013)}]{cuturi2013sinkhorn}
Cuturi, M., 2013. Sinkhorn distances: Lightspeed computation of optimal
  transport. Advances in neural information processing systems 26.

\bibitem[{DiMarino and Gerolin(2019)}]{dimarinogerolin}
DiMarino, S., Gerolin, A., 2019. An optimal transport approach for the
  {S}chr{\"o}dinger bridge problem and convergence of {S}inkhorn algorithm.
  Journal of Scientific Computing 85.
\newline\urlprefix\url{https://api.semanticscholar.org/CorpusID:208139296}

\bibitem[{Egan et~al.(2022)Egan, Bourne, Cotter, Cullen, Pelloni, Roper, and
  Wilkinson}]{egan2022new}
Egan, C.~P., Bourne, D.~P., Cotter, C.~J., Cullen, M.~J., Pelloni, B., Roper,
  S.~M., Wilkinson, M., 2022. A new implementation of the geometric method for
  solving the {E}ady slice equations. Journal of Computational Physics 469,
  111542.

\bibitem[{Ekeland and Temam(1976)}]{ekeland1976convex}
Ekeland, I., Temam, R., 1976. Convex Analysis and Variational Problems. Vol.~1
  of Studies in Mathematics and its Applications. North-Holland, Amsterdam.

\bibitem[{Feydy et~al.(2019)Feydy, S{\'e}journ{\'e}, Vialard, Amari,
  Trouv{\'e}, and Peyr{\'e}}]{Feydy}
Feydy, J., S{\'e}journ{\'e}, T., Vialard, F.-X., Amari, S.-i., Trouv{\'e}, A.,
  Peyr{\'e}, G., 16--18 Apr 2019. Interpolating between optimal transport and
  {MMD} using {Sinkhorn} divergences. In: Chaudhuri, K., Sugiyama, M. (Eds.),
  Proceedings of the 22nd International Conference on Artificial Intelligence
  and Statistics (AISTATS 2019). Vol.~89 of Proceedings of Machine Learning
  Research. PMLR, pp. 2681--2690.

\bibitem[{Genevay et~al.(2018)Genevay, Peyre, and Cuturi}]{genevay}
Genevay, A., Peyre, G., Cuturi, M., 09--11 Apr 2018. Learning generative models
  with {S}inkhorn divergences. In: Storkey, A., Perez-Cruz, F. (Eds.),
  Proceedings of the Twenty-First International Conference on Artificial
  Intelligence and Statistics. Vol.~84 of Proceedings of Machine Learning
  Research. PMLR, pp. 1608--1617.
\newline\urlprefix\url{https://proceedings.mlr.press/v84/genevay18a.html}

\bibitem[{Gigli(2011)}]{gigli2011mintytrick}
Gigli, N., Mar. 2011. On {H}\"{o}lder continuity-in-time of the optimal
  transport map towards measures along a curve. Proceedings of the Edinburgh
  Mathematical Society 54~(2), 401--409.
\newline\urlprefix\url{https://doi.org/10.1017/s001309150800117x}

\bibitem[{Hoskins(1971)}]{hoskins1971atmospheric}
Hoskins, B.~J., 1971. Atmospheric frontogenesis models: Some solutions.
  Quarterly Journal of the Royal Meteorological Society 97~(412), 139--153.

\bibitem[{Hoskins(1975)}]{hoskins1975geostrophic}
Hoskins, B.~J., 1975. The geostrophic momentum approximation and the
  semi-geostrophic equations. Journal of Atmospheric Sciences 32~(2), 233--242.

\bibitem[{Kitagawa et~al.(2019)Kitagawa, M{\'e}rigot, and
  Thibert}]{kitagawa2019convergence}
Kitagawa, J., M{\'e}rigot, Q., Thibert, B., 2019. Convergence of a {N}ewton
  algorithm for semi-discrete optimal transport. Journal of the European
  Mathematical Society 21~(9), 2603--2651.

\bibitem[{Lavier(2024)}]{lavier2024semi}
Lavier, T., 2024. A semi-discrete optimal transport scheme for the 3{D}
  incompressible semi-geostrophic equations. arXiv preprint arXiv:2411.00575.

\bibitem[{Li and Nochetto(2020)}]{liNochetto2020stability}
Li, W., Nochetto, R.~H., Jul. 2020. Quantitative stability and error estimates
  for optimal transport plans. {IMA} Journal of Numerical Analysis 41~(3),
  1941--1965.
\newline\urlprefix\url{https://doi.org/10.1093/imanum/draa045}

\bibitem[{M{\'e}rigot(2011)}]{merigot}
M{\'e}rigot, Q., 2011. {A multiscale approach to optimal transport}. {Computer
  Graphics Forum} 30~(5), 1584--1592.
\newline\urlprefix\url{https://hal.archives-ouvertes.fr/hal-00604684}

\bibitem[{Mokhtari et~al.(2020)Mokhtari, Ozdaglar, and Pattathil}]{Mokhtari20}
Mokhtari, A., Ozdaglar, A., Pattathil, S., 26--28 Aug 2020. A unified analysis
  of extra-gradient and optimistic gradient methods for saddle point problems:
  Proximal point approach. In: Chiappa, S., Calandra, R. (Eds.), Proceedings of
  the Twenty Third International Conference on Artificial Intelligence and
  Statistics. Vol. 108 of Proceedings of Machine Learning Research. PMLR, pp.
  1497--1507.
\newline\urlprefix\url{https://proceedings.mlr.press/v108/mokhtari20a.html}

\bibitem[{Pal(2019)}]{pal2019difference}
Pal, S., 2019. On the difference between entropic cost and the optimal
  transport cost. arXiv: Probability.
\newline\urlprefix\url{https://api.semanticscholar.org/CorpusID:168170192}

\bibitem[{Peyr{\'e} et~al.(2019)Peyr{\'e}, Cuturi, et~al.}]{PeyreBook}
Peyr{\'e}, G., Cuturi, M., et~al., 2019. Computational optimal transport: with
  applications to data science. Foundations and Trends{\textregistered} in
  Machine Learning 11~(5-6), 355--607.

\bibitem[{Pooladian et~al.(2022)Pooladian, Cuturi, and
  Niles-Weed}]{pooladian22}
Pooladian, A.-A., Cuturi, M., Niles-Weed, J., 2022. Debiaser beware: pitfalls
  of centering regularized transport maps.
\newline\urlprefix\url{https://arxiv.org/abs/2202.08919}

\bibitem[{Sarrazin(2022)}]{sarrazin2022lagrangian}
Sarrazin, C., 2022. Lagrangian discretization of variational problems in
  {W}asserstein spaces. Ph.D. thesis, Universit{\'e} Paris-Saclay.

\bibitem[{{Schmitzer}(2019)}]{Schmitzer}
{Schmitzer}, B., 2019. Stabilized sparse scaling algorithms for entropy
  regularized transport problems. SIAM J. Sci. Comput. 41~(3).

\bibitem[{Shutts and Cullen(1987)}]{shutts1987parcel}
Shutts, G., Cullen, M., 1987. Parcel stability and its relation to
  semigeostrophic theory. Journal of Atmospheric Sciences 44~(9), 1318--1330.

\end{thebibliography}

%%%%%%%%%%%%%%%%%%%%%%%%%%%%%%%%%%%%%%%%%%%%%%%%%%%%%%%%%%%%%%%%%%%%%%%%%%%%%%%%%%%%%%%%%%%%%%%%%%%%%%%%%%%%%%%%%%%%%%%%%%%%%%%%%
% Jacob: I've added an appendix
\appendix
\renewcommand{\thesection}{S\arabic{section}}
\renewcommand{\thesubsection}{S\arabic{section}.\arabic{subsection}}
\renewcommand{\thesubsubsection}{S\arabic{section}.\arabic{subsection}.\arabic{subsubsection}}

\renewcommand{\thefigure}{S\arabic{figure}}
\renewcommand{\thetable}{S\arabic{table}}
\renewcommand{\theequation}{S\arabic{equation}}

\setcounter{section}{0}
\setcounter{figure}{0}
\setcounter{table}{0}
\setcounter{equation}{0}

\end{document}